\documentclass[reqno,11pt]{amsart}

\usepackage{bbm}
\usepackage{amssymb,mathrsfs,amsfonts}
\usepackage[colorlinks=false,pdfborder={0 0 0} %,pagebackref
]{hyperref}

\usepackage[latin1]{inputenc}

\newtheorem{theorem}{Theorem}[section]
\newtheorem{coro}[theorem]{Corollary}
\newtheorem{lemma}[theorem]{Lemma}
\newtheorem{fact}[theorem]{Fact}

\theoremstyle{definition}
\newtheorem{definition}[theorem]{Definition}
\newtheorem{remark}[theorem]{Remark}

\newenvironment{spezbew}[1]{\begin{proof}[Proof of #1]}{\end{proof}}

\newcommand\cf{\mathbbm{1}}
\newcommand\C{\mathbb{C}}

\renewcommand\H{\mathbb{H}}
\newcommand\N{\mathbb{N}}

\newcommand\R{\mathbb{R}}
\newcommand\Z{\mathbf{Z}}
\renewcommand\a{\mathfrak{a}}

\newcommand\eps{\varepsilon}

\newcommand\Om{\Omega}
\newcommand\om{\omega}
\newcommand\wt{\widetilde}
\newcommand\la{\lambda}
\newcommand\al{\alpha}
\newcommand\ga{\gamma}

\newcommand\si{\sigma}
\newcommand\Si{\Sigma}
\newcommand\del{\delta}
\newcommand\ph{\varphi}
\newcommand\leqC{\lesssim}

\newcommand\dom{\mathcal{D}}
\newcommand\ov{\overline}

\newcommand{\dist}{\operatorname{dist}}
\newcommand{\supp}{\operatorname{supp}}
\renewcommand{\Re}{\operatorname{Re}}

\newenvironment{Quote}%
{\begin{list}{}{%
\setlength{\leftmargin}{1em}%
\setlength{\rightmargin}{1em}}%
\item[]\ignorespaces}%
{\unskip\end{list}}

\begin{document}

\title{Spectral multiplier theorems of Hörmander type\linebreak
on Hardy and Lebesgue spaces}

\author{Peer Christian Kunstmann}
\address{Department of Mathematics, Karlsruhe Institute of Technology (KIT),
Kaiserstr.\ 89, 76128 Karlsruhe, Germany}
\email{\href{mailto:peer.kunstmann@kit.edu}{peer.kunstmann@kit.edu},
\href{mailto:matthias.uhl@kit.edu}{matthias.uhl@kit.edu}}
\author{Matthias Uhl}

%\date{\today}



\begin{abstract}
Let $X$ be a space of homogeneous type and let $L$ be an injective,
non-negative, self-adjoint operator on $L^2(X)$ such that the
semigroup generated by $-L$ fulfills Davies-Gaffney estimates of
arbitrary order. We prove that the operator $F(L)$, initially
defined on $H^1_L(X)\cap L^2(X)$, acts as a bounded linear operator
on the Hardy space $H^1_L(X)$ associated with $L$ whenever $F$ is a
bounded, sufficiently smooth function. Based on this result,
together with interpolation, we establish Hörmander type spectral
multiplier theorems on Lebesgue spaces for non-negative,
self-adjoint operators satisfying generalized Gaussian estimates in
which the required differentiability order is relaxed compared to
all known spectral multiplier results.
\end{abstract}

\maketitle
\tableofcontents

\numberwithin{equation}{section}

\section{Introduction}

Let $L$ be a non-negative, self-adjoint operator on the Hilbert
space $L^2(X)$, where $X$ is a $\sigma$-finite measure space. If $E_L$
denotes the resolution of the identity associated with $L$, the
spectral theorem asserts that the operator
\begin{align}\label{specthmL2}
F(L):=\int_0^\infty F(\la)\,dE_L(\la)
\end{align}
is well defined and acts as a bounded linear operator on $L^2(X)$
whenever $F\colon[0,\infty)\to\C$ is a bounded Borel function.
Spectral multiplier theorems provide regularity assumptions on $F$
which ensure that the operator $F(L)$ extends from $L^p(X)\cap
L^2(X)$ to a bounded linear operator on $L^p(X)$ for all $p$ ranging
in some symmetric interval containing $2$.

In 1960, L.\ H{\"o}rmander addressed this question for the Laplacian
$L=-\Delta$ on $X=\R^D$ during his studies on the boundedness of
Fourier multipliers on $\R^D$. His famous Fourier multiplier theorem
(\cite[Theorem~2.5]{Hor60}) states that the operator $F(-\Delta)$ is
of weak type $(1,1)$ whenever $F\colon[0,\infty)\to\C$ is a bounded
Borel function such that
\begin{align}\label{HC}
\sup_{n\in\Z}\|\om F(2^n\cdot)\|_{H_2^s}<\infty
\end{align}
for some $s>D/2$. Here and in the following
$\om\in C_c^\infty(0,\infty)$ is a non-negative function such that
\begin{align*}
\supp\om\subset(1/4,1) \qquad\mbox{and}\qquad
\sum_{n\in\Z}\om(2^{-n}\la)=1\quad\mbox{for all $\la>0$}\,.
\end{align*}
As a consequence, $F(-\Delta)$ is bounded on $L^p(\R^D)$ for every
$p\in(1,\infty)$. Note that the so-called {\em H{\"o}rmander
condition} (\ref{HC}) does not depend on the special choice of
$\om$. By considering imaginary powers $(-\Delta)^{i\tau}$ for
$\tau\in\R$, M.\ Christ (\cite[p.~73]{Christ}) observed that the
regularity order in H{\"o}rmander's statement cannot be improved
beyond $D/2$. This means that for any $s<D/2$ there exists a bounded
Borel function $F\colon[0,\infty)\to\C$ such that the H{\"o}rmander
condition (\ref{HC}) holds, but $F(-\Delta)$ does not act as a
bounded operator on $L^p(\R^D)$ for the whole range
$p\in(1,\infty)$.

H{\"o}rmander's multiplier theorem was generalized, on the one hand,
to other spaces than $\R^D$ and, on the other hand, to more general
operators than the Laplacian. The development began in the early
1990's. G.\ Mauceri and S.\ Meda (\cite{MaMe}) and M.\ Christ
(\cite{Christ}) extended the result to homogeneous Laplacians on
stratified nilpotent Lie groups. Further generalizations were
obtained by G.\ Alexopoulos (\cite{Alexo}) who showed in the setting
of connected Lie groups of polynomial volume growth a corresponding
statement for the left invariant sub-Laplacian which was in turn
extended by W.\ Hebisch (\cite{Hebisch}) to integral operators with
kernels decaying polynomially away from the diagonal. More
historical remarks about spectral multiplier theorems can be found
e.g.\ in \cite{DOS} and the references therein.

The results in \cite{DOS} due to X.T.\ Duong, E.M.\ Ouhabaz, and A.\
Sikora marked an important step toward the study of more general
operators. In the abstract framework of (subsets of) spaces of
homogeneous type $(X,d,\mu)$ with dimension $D>0$ they investigated
non-negative, self-adjoint operators $L$ on $L^2(X)$ which satisfy
{\em pointwise Gaussian estimates}, i.e.\ the semigroup
$(e^{-tL})_{t>0}$ generated by $-L$ can be represented as integral
operators
$$
e^{-tL}f(x)=\int_Xp_t(x,y)f(y)\,d\mu(y)
\qquad(f\in L^2(X), t>0, \mu\mbox{-a.e.\ }x\in X)
$$
and the kernels $p_t\colon X\times X\to\C$ enjoy the following
pointwise upper bound
\begin{align}\label{GE}
|p_t(x,y)|\leq
C\,\mu(B(x,t^{1/m}))^{-1}\exp\Biggl(-b\biggl(\frac{d(x,y)}{t^{1/m}}
\biggr)^{\frac m{m-1}}\Biggr)
\end{align}
for all $t>0$ and all $x,y\in X$, where $b,C>0$ and $m\geq2$ are
constants independent of $t,x,y$ and $B(x,r):=\{y\in X:\,d(x,y)<r\}$
denotes the open ball in $X$ with center $x\in X$ and radius $r>0$. 
Under these hypotheses X.T.\ Duong, E.M.\ Ouhabaz, and A.\ Sikora 
proved that the operator $F(L)$ is of weak type $(1,1)$ whenever
$F\colon[0,\infty)\to\C$ is a bounded Borel function such that
$\sup_{n\in\Z}\|\om F(2^n\cdot)\|_{H_2^s}<\infty$ for some
$s>(D+1)/2$. Consequently, $F(L)$ is then bounded on $L^p(X)$ for 
all $p\in(1,\infty)$.

However, the price for the generality lies in the requirement of an
additional $1/2$ in the regularity order of the H{\"o}rmander
condition. Unfortunately, sharp results as for the Laplacian are
unknown at present time. In the general situation it is only known
that the regularity assumption $s>D/2+1/6$ cannot be weakened as an
example in \cite{Tha89} by S.\ Thangavelu shows.

In order to get better multiplier results in the general situation
as well, X.T.\ Duong, E.M.\ Ouhabaz, and A.\ Sikora introduced the
so-called {\em Plancherel condition} (\cite[(3.1)]{DOS}) which means
the following: There exist $C>0$ and $q\in[2,\infty]$ such that for
all $R>0$, $y\in X$, and all bounded Borel functions
$F\colon[0,\infty)\to\C$ with $\supp F\subset[0,R]$
\begin{align}\label{Plancherel-DOS-Original}
\int_X\bigl|K_{F(\sqrt[m]L)}(x,y)\bigr|^2\,d\mu(x)\leq
C\mu(B(y,1/R))^{-1}\|F(R\cdot)\|_{L^q}^2\,,
\end{align}
where $K_{F(\sqrt[m]L)}\colon X\times X\to\C$ denotes the kernel of
the integral operator $F(\sqrt[m]L)$. The result of X.T.\ Duong,
E.M.\ Ouhabaz, and A.\ Sikora reads as follows (\cite[Theorem
3.1]{DOS}):
\begin{Quote}
{\em Let $(X,d,\mu)$ be a space of homogeneous type with dimension
$D$ and $L$ be a non-negative, self-adjoint operator on $L^2(X)$
which satisfies pointwise Gaussian estimates. Suppose that the
Plan\-cherel condition holds for some $q\in[2,\infty]$ and that
$F\colon[0,\infty)\to\C$ is a bounded Borel function with
$\sup_{n\in\Z}\|\om F(2^n\cdot)\|_{H_q^s}<\infty$ for some $s>D/2$.
Then the operator $F(L)$ is of weak type $(1,1)$ and thus bounded 
on $L^p(X)$ for all $p\in(1,\infty)$.}
\end{Quote}
Here, we have set $H^s_\infty:=C^s$, the space of H\"older continuous
functions.

Sometimes it is not clear whether, or even not true that, a
non-negative, self-adjoint operator on $L^2(X)$ admits \emph{pointwise}
Gaussian estimates and therefore the above results cannot be applied.
This occurs, for example, for Schr{\"o}dinger operators with bad
potentials (\cite{Schreieck}) or elliptic operators of higher order
with bounded measurable coefficients (\cite{D97}). Nevertheless, it
is often possible to show a weakened version of (\ref{GE}),
so-called \emph{generalized} Gaussian estimates.
\begin{definition}
Let $1\leq p\leq2\leq q\leq\infty$ and $m\geq2$. A non-negative,
self-adjoint operator $L$ on $L^2(X)$ is said to satisfy {\em
generalized Gaussian $(p,q)$-estimates of order $m$} if there are
constants $b,C>0$ such that
\begin{align}\label{GGE}
\bigl\|\cf_{B(x,t^{1/m})}e^{-tL}\cf_{B(y,t^{1/m})}\bigr\|_{L^p\to L^q}
\leq
C\,\mu(B(x,t^{1/m}))^{-(\frac1{p}-\frac1{q})}\exp\Biggl(
-b\biggl(\frac{d(x,y)}{t^{1/m}}\biggr)^{\frac m{m-1}}\Biggr)
\end{align}
for all $t>0$ and all $x,y\in X$. In this case, we will use the
shorthand notation GGE$_m(p,q)$. If $L$ satisfies GGE$_m(2,2)$, then
we also say that $L$ enjoys {\em Davies-Gaffney estimates} of order
$m$ and just write DG$_m$. Here, $\cf_{E_1}$ denotes the
characteristic function of the set $E_1$ and
$\|\cf_{E_1}e^{-tL}\cf_{E_2}\|_{L^p\to L^q}$ is defined via
$\sup_{\|f\|_{L^p}\leq1}\|\cf_{E_1}\cdot
e^{-tL}(\cf_{E_2}f)\|_{L^q}$ for Borel sets $E_1,E_2\subset X$.
\end{definition}
In the case $(p,q)=(1,\infty)$, this definition covers pointwise
Gaussian estimates (cf.\ \cite[Proposition~2.9]{BK1}).

In 2003, S.\ Blunck (\cite[Theorem~1.1]{B}) showed a spectral
multiplier theorem for non-negative, self-adjoint operators $L$ on
$L^2(X)$ satisfying GGE$_m(p_0,p_0')$ for some $p_0\in[1,2)$, where
$1/p_0+1/p_0'=1$. It guarantees that the operator $F(L)$ is of weak
type $(p_0,p_0)$ if $F\colon[0,\infty)\to\C$ is a bounded Borel
function such that $\sup_{n\in\Z}\|\om F(2^n\cdot)\|_{H_2^s}<\infty$
holds for some $s>(D+1)/2$. In particular, $F(L)$ is then bounded on
$L^p(X)$ for all $p\in(p_0,p_0')$.

Here, the required regularity order in the H{\"o}rmander condition
for getting a weak type $(p_0,p_0)$-bound is the same as needed for
the weak type $(1,1)$-bound in the corresponding statement for
operators enjoying pointwise Gaussian estimates. The proof of S.\
Blunck relies on the weak type $(p_0,p_0)$ criterion due to S.\
Blunck and the first named author (\cite[Theorem~1.1]{BK2}) and it
seems to be impossible to weaken the regularity assumptions with
this approach directly. However, since for boundedness of $F(L)$ on $L^2(X)$
no regularity of $F$ is needed, one expects, motivated by
interpolation, $s>(D+1)(1/p_0-1/2)$ instead of $s>(D+1)/2$ as a
sufficient regularity assumption in the H{\"o}rmander condition when
one is interested in boundedness of $F(L)$ in $L^{p}(X)$ for all
$p\in(p_0,p_0')$.

In order to establish such a multiplier result, we make use of Hardy
spaces which serve as a substitute of Lebesgue spaces. For our
purposes we shall consider specific Hardy spaces being associated
with the operator $L$, similarly to the way that the classical Hardy
spaces are adapted to the Laplacian. They were originally introduced
by P.\ Auscher, X.T.\ Duong and A.\ McIntosh in \cite{ADM} and
revised during the past ten years. We refer to the beginning of
Section \ref{secHardy} for a short survey on recent developments.
\begin{definition}
Let $L$ be an injective, non-negative, self-adjoint operator on
$L^2(X)$ which satisfies Davies-Gaffney estimates of order $m\geq2$.
Consider the {\em conical square function}
$$
Sf(x):=\biggl(\int_0^\infty\!\!\int_{B(x,t)}
|t^mL\,e^{-t^mL}f(y)|^2\,\frac{d\mu(y)}{|B(x,t)|}\,\frac{dt}{t}
\biggr)^{1/2}\qquad(f\in L^2(X),x\in X).
$$
For $p\in[1,2]$, the {\em Hardy space $H^p_{L}(X)$ associated with
the operator $L$} is said to be the completion of the set $\{f\in
L^2(X):\, Sf\in L^p(X)\}$ with respect to the norm
$$
\|f\|_{H^p_{L,S}}:=\|Sf\|_{L^p}\,.
$$
\end{definition}
By the spectral theorem, it is plain to see that $H^2_{L}(X)=L^2(X)$
with equivalent norms. Hardy spaces associated with $L$ are known to possess nice
properties, for example, they form a complex interpolation scale
(cf.\ Fact \ref{HardyInterpol}), coincide under the assumption of
GGE$_m(p_0,2)$ with $L^p(X)$ for all $p\in(p_0,2]$ (cf.\
Theorem~\ref{HpLp}) and allow spectral multiplier theorems even for 
all $p\in[1,p_0]$ (cf.\ Sections \ref{SpecMultHardy}, 
\ref{Chapter6Interpol}).

There is an equivalent characterization of the space $H^1_{L}(X)$ in
terms of a molecular decomposition (cf.\ Theorem~\ref{sqfct=mol}). 
In order to verify boundedness of an operator on
the Hardy space $H^1_{L}(X)$, one has just to understand the action
of the operator on an individual molecule. Such an idea is classical
in the more comfortable situation of an atomic decomposition and was
used by various authors for obtaining boundedness of spectral
multipliers on the Hardy space $H^1_{L}(X)$. For example, J.\
Dziuba\'nski (\cite{Dzi1}) showed a spectral multiplier theorem for
Schr{\"o}dinger operators and, later, J.\ Dziuba\'nski and M.\
Preisner (\cite{Dzi2}) established a generalization to arbitrary
operators satisfying pointwise Gaussian estimates of order $2$.
Recently, X.T.\ Duong and L.X.\ Yan (\cite{DY}) obtained boundedness
of spectral multipliers on the Hardy space $H_{L}^1(X)$ for
operators $L$ satisfying Davies-Gaffney estimates of order $2$.

All these authors confined their studies to operators satisfying
Davies-Gaffney estimates of order $2$ and used essentially that, in
this case, the validity of Davies-Gaffney estimates is equivalent to
the finite speed propagation property for the corresponding wave
equation (cf.\ e.g.\ \cite[Theorem~3.4]{Phragmen}). Hence one obtains
information on the support of the integral kernel of $\cos(t\sqrt
L)$ and this in turn entails information on the support of the
integral kernel of $F(\sqrt L)$. However, for general $m>2$, such a
relation to finite speed propagation properties fails. We develop the 
following spectral multiplier theorem on the Hardy space $H_{L}^1(X)$ 
for operators $L$ satisfying Davies-Gaffney estimates of {\em arbitrary} 
order $m\geq2$ (cf.\ Theorem~\ref{DY_Thm1.1}~a)). 

\begin{theorem}\label{mainresH1}
Let $(X,d,\mu)$ be a space of homogeneous type with dimension
$D$ and $L$ an injective, non-negative, self-adjoint operator on
$L^2(X)$ satisfying Davies-Gaffney estimates of order $m\geq2$.
If a bounded Borel function $F\colon[0,\infty)\to\C$ satisfies
$\sup_{n\in\Z}\|\om F(2^n\cdot)\|_{H_2^s}<\infty$ for some 
$s>(D+1)/2$, then $F(L)$ can be extended from $H^1_L(X)\cap L^2(X)$
to a bounded linear operator on $H^1_L(X)$.
\end{theorem}

Based on ideas in \cite{DY} by X.T.\ Duong and L.X.\ Yan,
we give a sufficient criterion for the boundedness of spectral 
multipliers on $H^1_L(X)$ (cf.\ Theorem~\ref{DY_Thm3.1}), which 
will be achieved by reducing the proof of the boundedness of 
$F(L)$ in $H^1_L(X)$ to the uniform boundedness of $F(L)a$ in 
$H_{L}^1(X)$ for every molecule $a$. In order to derive the above 
Hörmander type multiplier theorem on $H^1_L(X)$, we use suitable 
weighted norm estimates that generalize the tools prepared in 
\cite{DOS} and compensate for the lack of information on the support 
caused by the missing finite speed propagation property.

We also present an improved spectral multiplier result with an
adequate $L^2$-version of the Plancherel condition
(\ref{Plancherel-DOS-Original}) which also works for operators $L$
satisfying Davies-Gaffney estimates. In order to motivate our
replacement, we rewrite (\ref{Plancherel-DOS-Original}) as a norm
estimate for the operator $F(\sqrt[m]L)$ itself
\begin{align*}
\bigl\|F(\sqrt[m]L)\,\cf_{B(y,1/R)}\bigr\|_{L^1\to L^2}\leq
C\,|B(y,1/R)|^{-\frac12}\|F(R\cdot)\|_{L^q}
\end{align*}
for all $R>0$, $y\in X$, and all bounded Borel functions
$F\colon[0,\infty)\to\C$ with $\supp F\subseteq[0,R]$, where the
constants $C>0$ and $q\in[2,\infty]$ are independent of $R,y,F$.
Inspired by this observation, we introduce our substitute of the
Plancherel condition for operators $L$ which fulfill Davies-Gaffney
estimates of order $m\geq2$ as follows:
\begin{align}\label{DOS_PlancherelEinleitung}
\bigl\|F(\sqrt[m]L)\,\cf_{B(y,1/R)}\bigr\|_{L^2\to L^2}\leq
C\,\|F(R\cdot)\|_{L^q}
\end{align}
for all $R>0$, $y\in X$, and all bounded Borel functions
$F\colon[0,\infty)\to\C$ with $\supp F\subseteq[0,R]$, where the
constants $C>0$ and $q\in[2,\infty]$ are independent of $R,y,F$.
Having this replace\-ment of (\ref{Plancherel-DOS-Original}) at
hand, we are able to show the following result (cf.\
Theorem~\ref{Hardy_Plancherel}).

\begin{theorem}\label{Hardy_Plancherel_Introd}
Let $(X,d,\mu)$ be a space of homogeneous type with dimension
$D$ and $L$ an injective, non-negative, self-adjoint operator on
$L^2(X)$ for which Davies-Gaffney estimates of order $m\geq2$ hold.
Suppose that $L$ fulfills the Plancherel condition
\eqref{DOS_PlancherelEinleitung}. If $F\colon[0,\infty)\to\C$ is a 
bounded Borel function with 
$\sup_{n\in\Z}\|\om F(2^n\cdot)\|_{H_q^s}<\infty$ for some
$s>\max\{D/2,1/q\}$, then there exists a constant $C>0$ such 
that for all $f\in H^1_L(X)$
\begin{align*}
\|F(L)f\|_{H^1_L}
\leq
C\Bigl(\sup_{n\in\Z}\|\om F(2^n\cdot)\|_{H_q^s}+|F(0)|\Bigr)\|f\|_{H^1_L}\,.
\end{align*}
\end{theorem}

In the same way as for the original Plancherel condition 
\eqref{Plancherel-DOS-Original} the validity of its variant
\eqref{DOS_PlancherelEinleitung} for some $q\in[2,\infty)$ 
entails that the point spectrum of the considered operator $L$ is 
empty. We also present a version of the Plancherel condition that 
applies for operators with non-empty point spectrum as well (cf.\
Theorem~\ref{DY_Thm1.1_Plancherel}). The approach is similar to 
the one of \cite[Theorem 3.2]{DOS}.

Since the Plancherel condition \eqref{DOS_PlancherelEinleitung}
always holds for $q=\infty$ (cf.\ Lemma~\ref{DOS_Lem2_2}),
Theorem~\ref{Hardy_Plancherel_Introd} yields the following 
multiplier result (cf.\ Theorem~\ref{DY_Thm1.1}~b)), in which 
the same order of differentiability is required as in 
\cite[Theorem~1.1]{DY} (which covers the case $m=2$).
\begin{theorem}
Let $(X,d,\mu)$ be a space of homogeneous type with dimension
$D$ and $L$ be as in Theorem \ref{mainresH1}. If 
$F\colon[0,\infty)\to\C$ is a bounded Borel function with
$\sup_{n\in\Z}\|\om F(2^n\cdot)\|_{C^s}<\infty$ for some $s>D/2$,
then $F(L)$ extends to a bounded linear operator on the Hardy space
$H^1_L(X)$.
\end{theorem}

Having spectral multiplier theorems on $H_{L}^1(X)$ at hand, we
can prove spectral multiplier results on Lebesgue spaces for
operators satisfying generalized Gaussian estimates
GGE$_m(p_0,p_0')$ for some $p_0\in[1,2)$ and $m\geq2$. In a first
step we combine our multiplier results on the Hardy space
$H_{L}^1(X)$ with the interpolation procedure \cite[Corollary
4.84]{CK} that allows to interpolate the regularity order in the
H\"ormander condition as well. This yields multiplier results on
$H^p_{L}(X)$ for all $p\in[1,2]$ (cf.\ Theorem~\ref{mainresHardy}).
As the spaces $H^p_{L}(X)$ and $L^p(X)$ coincide for each
$p\in(p_0,2]$, we obtain spectral multiplier theorems on Lebesgue
spaces which read as follows (cf.\ Theorem~\ref{mainres}):

\begin{theorem}
Let $(X,d,\mu)$ be a space of homogeneous type with dimension
$D$ and $L$ be a non-negative, self-adjoint operator on $L^2(X)$ such
that generalized Gaussian estimates GGE$_m(p_0,p_0')$ hold for some
$p_0\in[1,2)$ and $m\geq2$.
\begin{enumerate}
\item[\bf a)]
For fixed $p\in(p_0,p_0')$ suppose that $s>(D+1)|1/p-1/2|$ and
$1/q<|1/p-1/2|$. Then, for every bounded Borel function
$F\colon[0,\infty)\to\C$ with $\sup_{n\in\Z}\|\om
F(2^n\cdot)\|_{H_q^s}<\infty$, the operator $F(L)$ is bounded on
$L^p(X)$.

\item[\bf b)]
Let $p\in(p_0,p_0')$ and $s>D|1/p-1/2|$. Then, for any bounded
Borel function $F\colon[0,\infty)\to\C$ with $\sup_{n\in\Z}\|\om
F(2^n\cdot)\|_{C^s}<\infty$, the operator $F(L)$ is bounded on
$L^p(X)$.

\item[\bf c)]
In addition, assume that $L$ fulfills the Plancherel condition
\eqref{DOS_PlancherelEinleitung} for some $q_0\in[2,\infty)$. Fix
$p\in(p_0,p_0')$. Let $s>\max\{D,2/q_0\}\,|1/p-1/2|$ and
$1/q<2/q_0\,|1/p-1/2|$. Then, for every bounded Borel function
$F\colon[0,\infty)\to\C$ with $\sup_{n\in\Z}\|\om
F(2^n\cdot)\|_{H_q^s}<\infty$, the operator $F(L)$ is bounded on
$L^p(X)$. 
\end{enumerate}
\end{theorem}

The statement a) improves the results \cite[Theorem~1.1]{B} of S.\
Blunck and \cite[Theorem 5.6]{KriPre} (see also 
\cite[Theorem~4.95]{CK}) of C.\ Kriegler in which the regularity 
orders $s>(D+1)/2$, $q=2$ and $s>D|1/p-1/2|+1/2$, $q=2$ were required,
respectively. However, \cite[Theorem~1.1]{B} also includes a weak 
type $(p_0,p_0)$ assertion for $F(L)$.

We emphasize that in the presence of pointwise Gaussian
estimates the aforementioned multiplier theorem due to X.T.\ Duong, 
E.M.\ Ouhabaz and A.\ Sikora in combination with interpolation would 
need the same order of regularity for $F$ as our main result for ensuring 
the boundedness of $F(L)$ on $L^p(X)$ for any $p\in(p_0,p_0')$.
Additionally, in the case $p_0=1$ the statement b) matches 
\cite[Theorem~3.1]{DOS} which is sharp in the sense that it includes 
the same regularity assumptions as needed for the Laplacian in 
Hörmander's multiplier theorem.

Recently, P.\ Chen, E.M.\ Ouhabaz, A.\ Sikora, and L.\ Yan obtained
a similar spectral multiplier result for operators $L$ satisfying
DG$_2$ in which the required order of differentiability is the same
as ours in c) provided that $L$ satisfies the so-called Stein-Tomas 
restriction type condition \cite[(ST$^q_{p_0,2}$)]{COSY} for some
$p_0\in[1,2)$ (cf.\ \cite[Theorem~4.1]{COSY}) . This
corresponds to the $L^{p_0}-L^2$-version of the Plancherel condition
\eqref{DOS_PlancherelEinleitung} and is thus more restrictive than
our assumption. The approach in \cite{COSY} makes no use of Hardy 
spaces, but uses the result of \cite{B}. On the other hand, the approach 
relies heavily on the finite speed propagation property and thus the 
method of proof is restricted to the case $m=2$.

\smallskip
Examples of operators to which our results apply but those in 
\cite{DY,COSY,Dzi1,Dzi2} are not applicable include higher order 
elliptic operators in divergence form with bounded complex-valued
coefficients on $\R^D$ (cf. \cite{D95,D97}).  These operators are given by forms 
$\a\colon H_2^k(\R^D)\times H_2^k(\R^D)\to\C$ of the type
$$
\a(u,v)=\int_{\R^D}\sum_{|\al|=|\beta|=k}a_{\alpha\beta}\,\partial^\alpha u\, \overline{\partial^\beta v}\,dx\,,
$$
where $a_{\alpha\beta}\colon\R^D\to\C$ are bounded and measurable
functions. We assume $a_{\alpha\beta}=\overline{a_{\beta\alpha}}$ 
for all $\alpha,\beta$ and Garding's inequality
$$
\a(u,u)\geq\delta\|\nabla^ku\|_{L^2}^2\qquad\mbox{for all }u\in H_2^k(\R^D)
$$
for some $\delta>0$, where $\|\nabla^ku\|_{L^2}^2:=
\sum_{|\alpha|=k}\|\partial^\alpha u\|_{L^2}^2$. Then $\a$ is a
closed symmetric form. The associated operator $L$ is defined by 
$u\in\dom(L)$ and $Lu=f$ if and only if $u\in H_2^k(\R^D)$ and
$\int_{\R^D}f\overline{v}\,dx =\a(u,v)$ for all $v\in H_2^k(\R^D)$. 
In the case $D>2k$, $L$ satisfies generalized Gaussian estimates
GGE$_m(p_0,p_0')$ with $m:=2k$ and $p_0:=2D/(m+D)$ (cf.\ \cite{D95}). 
It is well-known that $p_0$ is sharp in the sense that for any
$r\notin[p_0,p_0']$ there exists an operator $L$ in the given class
for which $e^{-tL}$ cannot be extended from $L^r(\R^D)\cap L^2(\R^D)$
to a bounded linear operator on $L^r(\R^D)$ for any $t>0$ (cf.\
e.g.\ \cite[Theorem~10]{D97}).

In another paper (\cite{KunUhl}) we discuss how spectral multiplier 
theorems of the type presented here apply to the second order
Maxwell operator with measurable coefficient matrices and the Stokes
operator with Hodge boundary conditions on bounded Lipschitz domains
in $\R^3$ as well as the time-dependent Lam\'e system equipped with
homogeneous Dirichlet boundary conditions.

\section{Preliminaries}

Throughout the whole article we assume that $(X,d,\mu)$ is a space
of homogeneous type with dimension $D$ as introduced in Section
\ref{spacehom} below. To avoid repetition, we skip this assumption
in all the subsequent statements.

We make use of the notation $B(x,r):=\{y\in X:\,d(y,x)<r\}$ for the
open ball in $X$ with center $x\in X$ and radius $r\geq0$. We shall
write $\la B(x,r)$ for the $\la$-dilated ball $B(x,\la r)$ and
$A(x,r,k)$ for the annular region $B(x,(k+1)r)\setminus B(x,kr)$,
where $k\in\N_0$, $\la>0$, $r>0$, and $x\in X$. The volume of a
Borel set $\Om\subset X$ will be denoted by $|\Om|:=\mu(\Om)$.

The symbol $\cf_E$ stands for the characteristic function of a Borel
set $E\subset X$, whereas the norm $\|\cf_{E_1}T\cf_{E_2}\|_{L^p\to L^q}$
is defined via $\sup_{\|f\|_{L^p}\leq1}\|\cf_{E_1}\cdot
T(\cf_{E_2}f)\|_{L^q}$ for a bounded linear operator $T$ on
$L^2(X)$, Borel sets $E_1,E_2\subset X$, and $1\leq p\leq
q\leq\infty$.

For $p\in[1,\infty]$ the conjugate exponent $p'$ is defined by
$1/p+1/p'=1$ with the usual convention $1/\infty:=0$.

For $q\in(1,\infty)$ and $s\geq0$, let $H^s_q$ denote the Bessel
potential space on $\R$, whereas $H^s_\infty$ stands for the Hölder
space $C^s$.

In the proofs, the letters $b,C$ denote generic positive constants
that are independent of the relevant parameters involved in the
estimates and may take different values at different occurrences. We
will often use the notation $a\leqC b$ if there exists a constant
$C>0$ such that $a\leq Cb$ for two non-negative expressions $a,b$;
$a\cong b$ stands for the validity of $a\leqC b$ and $b\leqC a$.

\subsection{Spaces of homogeneous type}
\label{spacehom}

We use the general framework of {\em spaces of homogeneous type} in
the sense of Coifman and Weiss \cite{CW}, i.e.\ $(X,d)$ is a non-empty metric
space endowed with a $\sigma$-finite regular Borel measure $\mu$
with $\mu(X)>0$ which satisfies the so-called {\em doubling 
condition}, that is, there exists a constant $C>0$ such that for all 
$x\in X$ and all $r>0$
\begin{align}\label{doubling}
\mu(B(x,2r))\leq C\,\mu(B(x,r))\,.
\end{align}
It is easy to see that the doubling condition \eqref{doubling}
entails the {\em strong homogeneity property}, i.e.\ the existence 
of constants $C,D>0$ such that for all $x\in X$, all $r>0$, and all
$\la\geq1$
\begin{align}\label{doublingDim}
\mu(B(x,\la r))\leq C\la^D\mu(B(x,r))\,.
\end{align}
In the sequel the value $D$ always refers to the constant in
(\ref{doublingDim}) which will be also called {\em dimension} of
$(X,d,\mu)$. Of course, $D$ is not uniquely determined and for any
$D'\geq D$ the inequality (\ref{doublingDim}) is still valid.
However, the smaller $D$ is, the stronger will be the multiplier
theorems we are able to obtain. Therefore, we are interested in
taking $D$ as small as possible.

There is a multitude of examples of spaces of homogeneous type. The
simplest one is the Euclidean space $\R^D$, $D\in\N$, equipped with
the Euclidean metric and the Lebesgue measure. Bounded open subsets
of $\R^D$ with Lipschitz boundary endowed with the Euclidean metric
and the Lebesgue measure are also spaces of homogeneous type.

We give a short review about well-known results concerning spaces of
homogeneous type and start with a simple but useful observation
which is a direct consequence of the doubling condition
(\ref{doubling}).
\begin{fact}\label{BKKugel}
There exists a constant $C>0$ such that for all $r>0$, $x\in X$, and
$y\in B(x,r)$
\begin{align*}
C^{-1}|B(y,r)|\leq|B(x,r)|\leq C\,|B(y,r)|\,.
\end{align*}
Consequently, it holds for any $r>0$ and any $x\in X$
\begin{align}\label{doubl}
C^{-1}\leq\int_{B(x,r)}\frac1{|B(y,r)|}\,d\mu(y)\leq C\,.
\end{align}
\end{fact}

An essential feature of spaces of homogeneous type is the validity
of covering results which mean that, as in the Euclidean setting,
one can cover a ball of radius $r$ by balls of radius $s$ and their
number is bounded from above by a term only involving the ratio
$r/s$ and the constants in (\ref{doublingDim}) whenever $r\geq s>0$.
\begin{lemma}\label{ueberdeckung}
For each $r\geq s>0$ and $y\in X$, there exist finitely many points
$y_1,\ldots,y_K$ in $B(y,r)$ such that
\begin{enumerate}
\item[\em (i)]
$d(y_j,y_k)>s/2$ for all $j,k\in\{1,\ldots,K\}$ with $j\ne k$;
\item[\em (ii)]
$B(y,r)\subseteq\bigcup_{k=1}^KB(y_k,s)$;
\item[\em (iii)]
$K\leqC(r/s)^D$;
\item[\em (iv)]
each $x\in B(y,r)$ is contained in at most $M$ balls $B(y_k,s)$,
where $M$ depends only on the constants in (\ref{doublingDim}) and
is independent of $r,s,x,y$.
\end{enumerate}
\end{lemma}
The existence of $y_1,\ldots,y_K\in B(y,r)$ with the properties (i)
and (ii) is well-known (cf.\ e.g.\ \cite[Lemmas~6.1,~6.2]{AuMar} or
\cite[pp.~68~ff.]{CW}). It can be easily shown that (iii) and (iv)
are valid for such a family of points.

\subsection{Off-diagonal estimates}

We collect some properties of two-ball estimates in the next
statement which are proved in \cite[Proposition~2.1]{BKLeg}.
\begin{fact}\label{GGEequiv}
Let $1\leq p\leq q\leq\infty$, $r>0$, $\om>1$, and
$g(\la):=Ce^{-b\la^\om}$ for some constants $b,C>0$. Suppose
that $T$ is a bounded linear operator on $L^2(X)$. Then the
following assertions are equivalent:
\begin{enumerate}
\item[\bf a)]
For all $x,y\in X$, it holds
$$
\bigl\|\cf_{B(x,r)}T\cf_{B(y,r)}\bigr\|_{L^p\to L^q}
\leq
|B(x,r)|^{-(\frac1p-\frac1q)}g\biggl(\frac{d(x,y)}r\biggr)\,.
$$
\item[\bf b)]
For all $x,y\in X$ and all $u,v\in[p,q]$ with $u\leq v$, it holds
$$
\bigl\|\cf_{B(x,r)}T\cf_{B(y,r)}\bigr\|_{L^u\to L^v}
\leq
|B(x,r)|^{-(\frac1u-\frac1v)}g\biggl(\frac{d(x,y)}r\biggr)\,.
$$
\item[\bf c)]
For all $x\in X$ and all $k\in\N$, it holds
$$
\bigl\|\cf_{B(x,r)}T\cf_{A(x,r,k)}\bigr\|_{L^p\to L^q}
\leq
|B(x,r)|^{-(\frac1p-\frac1q)}g(k)\,.
$$
\item[\bf d)]
For all balls $B_1,B_2\subset X$ and all $\al,\beta\geq0$ with
$\al+\beta=\frac1p-\frac1q$, it holds
$$
\bigl\|\cf_{B_1}v_r^\al Tv_r^\beta\cf_{B_2}\bigr\|_{L^p\to L^q}
\leq
g\biggl(\frac{\dist(B_1,B_2)}r\biggr)\,,
$$
where $\dist(B_1,B_2):=\inf\{d(x,y):\,x\in B_1, y\in B_2\}$ and
$v_r(x):=|B(x,r)|$ for $x\in X$.
\end{enumerate}
This statement is written modulo identification of $g$ and $\wt g$,
where $\wt g(\la)=ag(c\la)$ for some constants $a,c>0$ independent
of $r$, $\om$, $T$.
\end{fact}
Since the estimate stated in c) involves an annular set $A(x,r,k)$,
we call bounds of this kind {\em estimates of annular type}.

A very useful feature of generalized Gaussian estimates is that they
can be extended from real times $t>0$ to complex times $z\in\C$ with
$\Re z>0$. The following result is taken from
\cite[Theorem~2.1]{BGGE} whose proof relies on the
Phragm\'en-Lindelöf theorem.
\begin{fact}\label{GGEkomplex}
Let $m\geq2$, $1\leq p\leq2\leq q\leq\infty$, and $L$ be a
non-negative, self-adjoint operator on $L^2(X)$. Assume that there
are constants $b,C>0$ such that for any $t>0$ and $x,y\in X$
\begin{align*}
\bigl\|\cf_{B(x,t^{1/m})}e^{-tL}\cf_{B(y,t^{1/m})}\bigr\|_{L^p\to L^q}
\leq C\,|B(x,t^{1/m})|^{-(\frac1{p}-\frac1{q})} \exp
\Biggl(-b\biggl(\frac{d(x,y)}{t^{1/m}}\biggr)^{\frac
m{m-1}}\Biggr)\,.
\end{align*}
Then there exist constants $b',C'>0$ such that for all $x,y\in X$
and all $z\in\C$ with $\Re z>0$
$$
\bigl\|\cf_{B(x,r_z)}e^{-zL}\cf_{B(y,r_z)}\bigr\|_{L^p\to L^q} \leq
C'\,|B(x,r_z)|^{-(\frac1{p}-\frac1{q})} \biggl(\frac{|z|}{\Re
z}\biggr)^{D(\frac1p-\frac1{q})}\exp
\Biggl(-b'\biggl(\frac{d(x,y)}{r_z}\biggr)^{\frac m{m-1}}\Biggr)\,,
$$
where $r_z:=(\Re z)^{1/m-1}|z|$.
\end{fact}

Here the radius of the balls in the above two-ball estimate for
$e^{-zL}$ depends on the value of $z$. The next lemma provides
two-ball estimates with balls of arbitrary radius $r>0$ by the cost
of an additional factor involving the ratio of $r$ and $r_z$ as well
as the dimension of the underlying space of homogeneous type. Also a
corresponding version for estimates of annular type is given. We
postpone the proof to Section \ref{proofs}.
\begin{lemma}\label{GGEkomplexbelRadius}
Suppose that the assumptions of Fact \ref{GGEkomplex} are fulfilled
and, as before, define $r_z:=(\Re z)^{1/m-1}|z|$ for each
$z\in\C$ with $\Re z>0$.
\begin{enumerate}
\item[\bf a)]
There exist constants $b',C'>0$ such that for all $r>0$, $x,y\in X$,
and $z\in\C$ with $\Re z>0$
\begin{align*}
&\bigl\|\cf_{B(x,r)}e^{-zL}\cf_{B(y,r)}\bigr\|_{L^p\to L^q}
\\&\quad\leq
C'\,|B(x,r)|^{-(\frac1{p}-\frac1{q})}
\biggl(1+\frac{r}{r_z}\biggr)^{D(\frac1{p}-\frac1{q})}
\biggl(\frac{|z|}{\Re z}\biggr)^{D(\frac1{p}-\frac1{q})}
\exp\Biggl(-b'\biggl(\frac{d(x,y)}{r_z}\biggr)^{\frac m{m-1}}\Biggr)\,.
\end{align*}

\item[\bf b)]
There exist constants $b'',C''>0$ such that for all $k\in\N$, $r>0$,
$x\in X$, and $z\in\C$ with $\Re z>0$
\begin{align*}
&\bigl\|\cf_{B(x,r)}e^{-zL}\cf_{A(x,r,k)}\bigr\|_{L^p\to L^q}
\\&\quad\leq
C''\,|B(x,r)|^{-(\frac1{p}-\frac1{q})}
\biggl(1+\frac{r}{r_z}\biggr)^{D(\frac1{p}-\frac1{q})}
\biggl(\frac{|z|}{\Re z}\biggr)^{D(\frac1{p}-\frac1{q})}k^D
\exp\Biggl(-b''\biggl(\frac{r}{r_z}\,k\biggr)^{\frac m{m-1}}\Biggr)\,.
\end{align*}
\end{enumerate}
\end{lemma}

In Section \ref{secHardy} we consider specific Hardy spaces
associated with an operator $L$. For defining and working with these
spaces it is enough to require a special form of two-ball estimates
on $L^2(X)$ for the semigroup $(e^{-tL})_{t>0}$ generated by $-L$,
so-called Davies-Gaffney estimates.
\begin{definition}
Let $m\geq2$. We say that a family $\{S_t:t>0\}$ of bounded linear
operators acting on $L^2(X)$ satisfies {\em Davies-Gaffney
estimates of order $m$} if there exist constants $b,C>0$ such that
for all $t>0$ and all $x,y\in X$
\begin{align}\label{DG}
\bigl\|\cf_{B(x,t^{1/m})}S_t\cf_{B(y,t^{1/m})}\bigr\|_{L^2\to L^2}
\leq
C\exp\Biggl(-b\biggl(\frac{d(x,y)}{t^{1/m}}\biggr)^{\frac m{m-1}}\Biggr)\,.
\end{align}
In order to indicate the validity of Davies-Gaffney estimates of
order $m$, we later use the abbreviation DG$_{m}$. If
$\{S_t:t>0\}=(e^{-tL})_{t>0}$ is a semigroup on $L^2(X)$ generated
by $-L$, we shall also say that $L$ satisfies Davies-Gaffney
estimates when the semigroup $(e^{-tL})_{t>0}$ enjoys this property.
\end{definition}

Estimates of the type (\ref{DG}) were first introduced by E.B.\
Davies (\cite{D92}) inspired by ideas of M.P.\ Gaffney
(\cite{Gaffney}). They hold naturally for many operators, including
large classes of self-adjoint, elliptic differential 
operators or Schrödinger operators with real-valued potentials (cf.\
e.g.\ \cite{Phragmen}). Davies-Gaffney estimates were extensively
studied in the recent series of papers \cite{AuMar06},
\cite{AuMar07a}, \cite{AuMar}, \cite{AuMar08} by P.\ Auscher and
J.M.\ Martell (see also \cite{Phragmen}, \cite{DL}, \cite{HLMMY}).
We mention that in the literature one usually finds a slightly
different definition of Davies-Gaffney estimates in which the
validity of (\ref{DG}) is required for all open subsets of $X$. It
is known that the definitions coincide for $m=2$ (cf.\ 
\cite[Lemma~3.1]{Phragmen}).

Finally, we quote a statement originally given in \cite[Proposition
3.1]{HLMMY} for operators satisfying DG$_2$. However, with some
minor modifications the proof can be adapted to include
Davies-Gaffney estimates of arbitrary order $m\geq2$ as well. For a
detailed proof we refer to Section~\ref{proofs}.

\begin{lemma}\label{DLProp2.3}
Let $m\geq2$ and $L$ be a non-negative, self-adjoint operator on
$L^2(X)$. If $L$ fulfills Davies-Gaffney estimates DG$_m$, then for
each $K\in\N$ the family of operators
$$
\bigl\{(tL)^{K}e^{-tL}:\,t>0\bigr\}
$$
satisfies also Davies-Gaffney estimates DG$_m$ with constants
depending only on $K$ and the constants in the doubling condition
(\ref{doublingDim}) and the Davies-Gaffney condition (\ref{DG}) for
the semigroup $(e^{-tL})_{t>0}$.
\end{lemma}

\section{Hardy spaces associated with operators}
\label{secHardy}

Quite recently, a theory of Hardy spaces associated with certain operators
was introduced, similar to the way that classical Hardy spaces are
adapted to the Laplacian. We refer to \cite{DL} for a survey on the
recent development and only mention that their origin lies in the
paper \cite{ADM} by P.\ Auscher, X.T.\ Duong, and A.\ McIntosh, who
defined the Hardy space $H^1_L(\R^D)$ associated with an operator
$L$ which has a bounded holomorphic functional calculus on
$L^2(\R^D)$ and for which the semigroup operators have a pointwise
Poisson upper bound. Afterwards, the assumptions on the associated
operator were relaxed. S.\ Hofmann and S.\ Mayboroda (\cite{HM})
defined Hardy spaces associated with second order divergence form
elliptic operators on $\R^D$ with complex coefficients. S.\ Hofmann,
G.Z.\ Lu, D.\ Mitrea, M.\ Mitrea, and L.X.\ Yan (\cite{HLMMY})
developed a theory of Hardy spaces adapted to non-negative,
self-adjoint operators $L$ on $L^2(X)$ which satisfy Davies-Gaffney
estimates in the setting of spaces of homogeneous type. X.T.\ Duong
and J.\ Li (\cite{DL}) considered even non-self-adjoint operators
and introduced Hardy spaces associated with operators which have a
bounded holomorphic functional calculus on $L^2(X)$ and generate an
analytic semigroup on $L^2(X)$ satisfying Davies-Gaffney estimates
of order $2$.

Throughout this section, let $L$ be an injective, non-negative,
self-adjoint operator on $L^2(X)$ which satisfies Davies-Gaffney
estimates DG$_m$ for some $m\geq2$. We summarize the most important
facts about Hardy spaces associated with $L$. For more details and
proofs of the statements, we refer to \cite{HM}, \cite{HMM},
\cite{HLMMY}, \cite{DL}, \cite{Frey-preprint}, \cite{AMR},
\cite{CaoYang}, and \cite{DDY}. The proofs given there carry over
with only minor changes to our more general setting.

\begin{definition}\label{DefHardy}
Let $p\in[1,2]$. Put $\psi_0(z):=ze^{-z}$, $z\in\C$, and consider 
the {\em conical square function}
$$
Sf(x):=\biggl(
\int_0^\infty\!\!\int_{B(x,t)}|\psi_0(t^mL)f(y)|^2\,\frac{d\mu(y)}{|B(x,t)|}\,
\frac{dt}{t}
\biggr)^{1/2}\qquad(f\in L^2(X),x\in X).
$$
The {\em Hardy space $H^p_{L}(X)$ associated with $L$} is defined to
be the completion of
$$
\bigl\{f\in L^2(X)\,:\, Sf\in L^p(X)\bigr\}
$$
with respect to the norm
$$
\|f\|_{H^p_{L}}:=\|S f\|_{L^p}\,.
$$
\end{definition}

The definition of Hardy spaces associated with operators is also
possible for $p\in(0,1)$ or $p\in(2,\infty)$. In addition, other
functions than $\psi_0$ can be considered. More information on this
can be found in the aforementioned literature.

By using Fubini's theorem and the spectral theorem it can be
verified that $H^2_{L}(X)=L^2(X)$ with equivalent norms. Additionally, 
the set $H_L^1(X)\cap L^2(X)$ is dense in $H_L^1(X)$. Note that in the 
special case of $X=\R^D$ and $L=-\Delta$ this definition yields the 
Hardy space $H^p(\R^D)$ as introduced by E.M.\ Stein and G.\ Weiss
(\cite{SW}). Similar to classical Hardy spaces, Hardy spaces 
associated with operators form a complex interpolation scale. This 
can be verified by viewing these spaces via the framework of tent 
spaces and by using the interpolation properties of tent spaces 
(cf.\ \cite[Lemma~4.20]{HMM}).

\begin{fact}\label{HardyInterpol}
Suppose that $1\leq p_0<p<p_1\leq2$ with $1/p=(1-\theta)/p_0+\theta/p_1$
for some $\theta\in(0,1)$. Then it holds
$$
\bigl[H^{p_0}_{L}(X),H^{p_1}_{L}(X)\bigr]_\theta=H^p_{L}(X)\,.
$$
\end{fact}

It is well-known that the classical Hardy space $H^1(\R^D)$
possesses an atomic decomposition. This property carries over to
Hardy spaces associated with injective, non-negative, self-adjoint
operators $L$ satisfying Davies-Gaffney estimates of order $2$ (cf.\
\cite[Theorem 4.1]{HLMMY}). Besides the atomic decomposition of tent
spaces, the proof in \cite{HLMMY} relies heavily on the equivalence
between the Davies-Gaffney estimates DG$_2$ for $L$ and the finite
speed propagation property for the corresponding wave equation
$Lu+u_{tt}=0$ (cf.\ e.g.\ \cite[Theorem~3.4]{Phragmen}).
Unfortunately, it is not possible to deduce a result similar to the
finite speed propagation property for operators $L$ that fulfill
DG$_m$ for some $m>2$ and thus it seems not to be clear whether an
atomic decomposition of $H^1_{L}(X)$ for these operators $L$ is
possible. Nevertheless, in the general situation one can decompose
the Hardy space $H^1_{L}(X)$ by considering molecules instead of
atoms.
\begin{definition}\label{molecule}
Let $M\in\N$ and $\eps>0$. A function $a\in L^2(X)$ is said to be an
{\em $(M,\eps,L)$-molecule} if there exist a function
$b\in\dom(L^M)$ and a ball $B\subset X$ with radius $r$ such that
\begin{enumerate}
\item[\em (i)] $a=L^Mb$\,;
\item[\em (ii)] for every $k\in\{0,1,\ldots,M\}$ and $j\in\N_0$, it
holds
\begin{align}\label{mol}
\bigl\|(r^mL)^kb\bigr\|_{L^2(U_j(B))}
\leq
r^{mM}2^{-j\eps}\mu(2^jB)^{-1/2}\,,
\end{align}
where the dyadic annuli $U_j(B)$ are defined by
\begin{align}\label{annuli}
U_0(B):=B
\qquad\mbox{and}\qquad
U_j(B):=2^jB\setminus2^{j-1}B\quad\mbox{for all }j\in\N\,.
\end{align}
\end{enumerate}
In this situation we sometimes refer to $a$ as being an
$(M,\eps,L)$-molecule associated with $B$.
\end{definition}

In the literature (cf.\ e.g.\ \cite{HLMMY}, \cite{DL}, \cite{HMM}) authors
mostly study the case when $m=2$ and typically use the terminology
``$(1,2,M,\eps)$-molecule associated with $L$'' instead of
$(M,\eps,L)$-molecule. Next, we give the definition of the molecular
Hardy spaces associated with $L$ (cf.\ e.g.\ \cite{Frey-preprint},
\cite{CaoYang}).

\begin{definition}
Fix $M\in\N$ and $\eps>0$. Let $f\in L^1(X)$. We call
$f=\sum_{j=0}^\infty\la_jm_j$ a {\em molecular
$(M,\eps,L)$-representation} of $f$ if $(\la_j)_{j\in\N_0}\in\ell^1$
is a numerical sequence, $m_j$ is an
$(M,\eps,L)$-molecule for any $j\in\N_0$, and the sum
$\sum_{j=0}^\infty\la_jm_j$ converges in $L^2(X)$. Define
$$
\H^1_{L,mol,M,\eps}(X):=\bigl\{f\in L^1(X)\,:\,
\mbox{$f$ has a molecular $(M,\eps,L)$-representation}\bigr\}
$$
with the norm given by
\begin{align*}
\|f\|_{H^1_{L,mol,M,\eps}} &:=
\inf\Biggl\{\sum_{j=0}^\infty|\la_j|\,:\,\sum_{j=0}^\infty\la_jm_j
\mbox{ is a molecular $(M,\eps,L)$-representation of $f$}\Biggr\}\,.
\end{align*}
The {\em molecular Hardy space $H^1_{L,mol,M,\eps}(X)$ associated
with $L$} is said to be the completion of $\H^1_{L,mol,M,\eps}(X)$
with respect to the norm $\|\cdot\|_{H^1_{L,mol,M,\eps}}$.
\end{definition}

As a direct consequence of the definition, we note that
$H^1_{L,mol,M_2,\eps}(X)\subset H^1_{L,mol,M_1,\eps}(X)$ for
$\eps>0$ and $M_1,M_2\in\N$ with $M_1\leq M_2$. In addition, the
Hardy space $H^1_{L,mol,M,\eps}(X)$ is contained in $L^1(X)$ because
the $L^1(X)$-norm of $(M,\eps,L)$-molecules is uniformly bounded by
a constant depending only on $\eps$ and the constants in the
doubling condition.

One can show the following characterization. For a proof, we refer
to \cite[Theorem 3.12]{DL} (see also \cite{CaoYang} for $X=\R^D$).
\begin{theorem}\label{sqfct=mol}
Assume that $M\in\N$ with $M>\frac{D}{2m}$ and $\eps\in(0,mM-D/2]$.
Then
$$
H^1_{L,mol,M,\eps}(X)=H^1_{L}(X)
$$
with equivalent norms
$$
\|f\|_{H^1_{L,mol,M,\eps}}\cong\|f\|_{H^1_{L}}\,,
$$
where implicit constants depend only on $\eps$, $M$ or the
constants in the Davies-Gaffney and the doubling condition.

In particular, every function $f\in H_L^1(X)\cap L^2(X)$ admits a
molecular $(M,\eps,L)$-represen\-tation.
\end{theorem}

A detailed examination of the proof due to X.T.\ Duong and J.\ Li
shows the following
\begin{coro}\label{MolH1}
Let $\eps>0$ and $M\in\N$ with $M>\frac{D}{2m}$. Then every
$(M,\eps,L)$-molecule $a$ belongs to $H^1_{L}(X)$ and there is a
constant $C>0$ depending only on $\eps, M$ and the constants in the
Davies-Gaffney (\ref{DG}) and the doubling condition
(\ref{doublingDim}) such that for all $(M,\eps,L)$-molecules $a$:
$$
\|a\|_{H^1_{L}}\leq C\,.
$$
\end{coro}

Thanks to $H^1_{L}(X)\subset L^1(X)$ and $H^2_{L}(X)=L^2(X)$, Fact
\ref{HardyInterpol} yields that $H^p_{L}(X)\subset L^p(X)$ for each
$p\in(1,2)$. The question under which assumptions on $L$ the reverse
inclusion holds is settled for the classical Hardy spaces
$H^p(\R^D)$. It is well-known that they can be identified with the
Lebesgue spaces $L^p(\R^D)$ for any $p\in(1,\infty)$ (see e.g.\
\cite[p.\ 220]{Stein}). However, if $L$ is an injective,
non-negative, self-adjoint operator on $L^2(\R^D)$ which satisfies
Davies-Gaffney estimates DG$_{m}$ for some $m\geq2$ and $p\in(1,2)$, 
then $H^p_{L}(\R^D)$ may or may not coincide with $L^p(\R^D)$ (see 
e.g.\ \cite[Proposition~9.1~(v),~(vi)]{HMM}, where Riesz transforms 
were studied).

P.\ Auscher, X.T.\ Duong, and A.\ McIntosh showed in \cite[Theorem
6]{ADM} that pointwise Gaussian estimates (\ref{GE}) imply
$H^p_{L}(\R^D)=L^p(\R^D)$ for all $p\in(1,2]$. By reasoning similar
to P.\ Auscher in \cite[Proposition 6.8]{AuscherOnNec}, who sketched
a proof in the case of second order divergence form operators on
$\R^D$, one can show a corresponding result for operators satisfying
only generalized Gaussian estimates. In the case $m=2$ this is
already stated in \cite[Proposition 9.1 (v)]{HMM}, with a reference to
\cite{AuscherOnNec} for the proof.

\begin{theorem}\label{HpLp}
Assume that $L$ is an injective, non-negative, self-adjoint operator
on $L^2(X)$ which fulfills generalized Gaussian estimates
GGE$_m(p_0,p_0')$ for some $p_0\in[1,2)$ and $m\geq2$. Then, for
each $p\in(p_0,2]$, the Hardy space $H^p_{L}(X)$ and the Lebesgue
space $L^p(X)$ coincide and their norms are equivalent.
\end{theorem}

By density it is enough to establish the estimate
$\|Sf\|_{L^p}\cong\|f\|_{L^p}$ for every $f\in L^p(X)\cap L^2(X)$.
This is divided into three steps. In a first step, which is the 
main work, one verifies that $\|Sf\|_{L^p}\leqC\|f\|_{L^p}$ for all 
$f\in L^p(X)\cap L^2(X)$ and $p\in(p_0,2]$. In a second step 
it is shown that this estimate is actually valid for any 
$p\in(2,p_0')$. In the final step three one can deduce the reverse
inequality $\|f\|_{L^p}\leqC\|S_{\psi}f\|_{L^p}$ for all $f\in L^p(X)\cap 
L^2(X)$ and $p\in(p_0,2]$ by a dualization argument based on 
the bound obtained in the second step.

The idea of the proof of the first step consists in establishing a
weak type $(p_0,p_0)$-estimate for the square function $S$. As
technical difficulties arise, which are caused by the definition of $S$ via an
area integral, one studies the properties of what may be called {\em
Littlewood-Paley-Stein $g^*_{\la}$-function adapted to $L$}
$$
g^*_{\la}(f)(x)
:=
\Biggl(\int_0^\infty\!\! \int_X\biggl(\frac{s^{1/m}}{d(x,y)+s^{1/m}}\biggr)^{D\la}
|sLe^{-sL}f(y)|^2\,\frac{d\mu(y)}{|B(x,s^{1/m})|}\,\frac{ds}{s}\Biggr)^{1/2}
$$
for $\la>0$, $x\in X$, and $f\in L^2(X)$. It turns out that
$g^*_{\la}$ is better suited than $S$ as far as Fubini arguments are
concerned because it contains an integral over the full space.
Since $g^*_{\la}$ controls $S$ for any $\la>1$, it suffices to verify
a weak type $(p_0,p_0)$-estimate for $g_\la$. A detailed proof can be 
found in \cite[Section 4.4]{Uhl}.

\section{Spectral multipliers on the Hardy space
\texorpdfstring{$\smash{H^1_L(X)}$}{H\^{}1(X)}}
\label{SpecMultHardy}

In this section, we formulate and prove Hörmander type spectral
multiplier results on $H^1_L(X)$, where $L$ is an injective, 
non-negative, self-adjoint operator on $L^2(X)$ which
satisfies Davies-Gaffney estimates of arbitrary order $m\geq2$. We
will state three versions, namely a more classical one, presented in
Theorem \ref{DY_Thm1.1}, and two including a Plancherel condition
which leads to weakened regularity assumptions on the involved
function, given in Theorem \ref{Hardy_Plancherel} and Theorem
\ref{DY_Thm1.1_Plancherel}.

\begin{theorem}\label{DY_Thm1.1}
Let $L$ be an injective, non-negative, self-adjoint operator on
$L^2(X)$ satisfying Davies-Gaffney estimates of order $m\geq2$.
\begin{enumerate}
\item[\bf a)]
If $s>(D+1)/2$ and $F\colon[0,\infty)\to\C$ is a bounded Borel 
function with
\begin{align}\label{HC_Bessel}
\sup_{n\in\Z}\|\om F(2^n\cdot)\|_{H_2^s}<\infty\,,
\end{align}
then $F(L)$ can be extended from $H^1_L(X)\cap L^2(X)$ to a bounded 
linear operator on $H^1_L(X)$. More precisely, there exists a 
constant $C>0$ such that
$$
\|F(L)\|_{H^1_L(X)\to H^1_L(X)}\leq C \Bigl(
\sup_{n\in\Z}\|\om F(2^n\cdot)\|_{H_2^s}+|F(0)|\Bigr)\,.
$$

\item[\bf b)]
If $s>D/2$ and $F\colon[0,\infty)\to\C$ is a bounded Borel function 
with
\begin{align}\label{HC_Hoelder}
\sup_{n\in\Z}\|\om F(2^n\cdot)\|_{C^s}<\infty\,,
\end{align}
then $F(L)$ extends to a bounded linear operator on the Hardy space
$H^1_L(X)$. To be more precise, there is a constant $C>0$ such that
\begin{align*}
\|F(L)\|_{H^1_L(X)\to H^1_L(X)}
\leq
C\Bigl(\sup_{n\in\Z}\|\om F(2^n\cdot)\|_{C^s}+|F(0)|\Bigr)\,.
\end{align*}
\end{enumerate}
\end{theorem}

In the special case $m=2$ the statement b) corresponds to
\cite[Theorem~1.1]{DY}.

\begin{theorem}\label{Hardy_Plancherel}
Let $L$ be an injective, non-negative, self-adjoint operator on
$L^2(X)$ for which Davies-Gaffney estimates of order $m\geq2$ hold.
Suppose that there exist $C>0$ and $q\in[2,\infty]$ such that for
any $R>0$, $y\in X$, and any bounded Borel function
$F\colon[0,\infty)\to\C$ with $\supp F\subset[0,R]$
\begin{align}\label{DOS_Plancherel_Hardy}
\bigl\|F(\sqrt[m]L)\,\cf_{B(y,1/R)}\bigr\|_{L^2\to L^2}\leq
C\,\|F(R\cdot)\|_{L^q}\,.
\end{align}
If $s>\max\{D/2,1/q\}$ and $F\colon[0,\infty)\to\C$ is a bounded
Borel function with
\begin{align*}
\sup_{n\in\Z}\|\om F(2^n\cdot)\|_{H_q^s}<\infty\,,
\end{align*}
then there exists a constant $C>0$ such that for all $f\in H^1_L(X)$
\begin{align*}
\|F(L)f\|_{H^1_L}
\leq
C\Bigl(\sup_{n\in\Z}\|\om F(2^n\cdot)\|_{H_q^s}+|F(0)|\Bigr)\|f\|_{H^1_L}\,.
\end{align*}
\end{theorem}

As already mentioned in the introduction, in general the assertion
of Theorem \ref{Hardy_Plancherel} is false without the Plancherel
condition (\ref{DOS_Plancherel_Hardy}). But
(\ref{DOS_Plancherel_Hardy}) always holds for $q=\infty$, as Lemma
\ref{DOS_Lem2_2} shows (cp.\ \cite[Lemma~2.2]{DOS} for a similar
result). Consequently, Theorem~\ref{DY_Thm1.1}~b) follows from
Theorem~\ref{Hardy_Plancherel}.

On the one hand,
(\ref{DOS_Plancherel_Hardy}) ensures that the class of functions for
which the multiplier result applies is extended. However, on the
other hand, the validity of (\ref{DOS_Plancherel_Hardy}) for some
$q\in[2,\infty)$ entails the emptiness of the point spectrum of $L$.
Indeed, according to the Plancherel condition
(\ref{DOS_Plancherel_Hardy}), one has for all $0\leq a\leq R$ and
$y\in X$
\begin{align*}
\bigl\|\cf_{\{a\}}(\sqrt[m]L)\,\cf_{B(y,1/R)}\bigr\|_{L^{p_0}\to L^2}
\leqC
|B(y,1/R)|^{-(\frac1{p_0}-\frac12)}\|\cf_{\{a\}}(R\cdot)\|_{L^q}
=
0
\end{align*}
and therefore $\cf_{\{a\}}(\sqrt[m]L)=0$. Due to
$\si(L)\subseteq[0,\infty)$, it follows that the point spectrum of
$L$ is empty. In order to treat operators with non-empty point
spectrum as well, one may introduce some variation of the Plancherel
condition (\ref{DOS_Plancherel_Hardy}). This approach originates in
\cite{CoSi} and was also used in \cite{DOS} or \cite{COSY}. For
$N\in\N$, $q\in[1,\infty)$, and a bounded Borel function
$F\colon\R\to\C$ with $\supp F\subseteq[-1,2]$ define the norm
$\|F\|_{N,q}$ via the formula
$$
\|F\|_{N,q}:=
\Biggl(\frac1N\sum_{k=1-N}^{2N}\sup_{\la\in[\frac{k-1}N,\frac
kN)}|F(\la)|^q\Biggr)^{1/q}\,.
$$

\begin{theorem}\label{DY_Thm1.1_Plancherel}
Let $L$ be an injective, non-negative, self-adjoint operator on
$L^2(X)$ satisfying DG$_m$ for some $m\geq2$. Fix $\kappa\in\N$ and
$q\in[2,\infty)$. Suppose that there is $C>0$ such that for any
$N\in\N$, $y\in X$, and any bounded Borel function $F\colon\R\to\C$
with $\supp F\subset[-1,N+1]$
\begin{align*}
\bigl\|F(\sqrt[m]L)\,\cf_{B(y,1/N)}\bigr\|_{L^2\to L^2}
\leq
C\,\|F(N\cdot)\|_{N^\kappa,q}\,.
\end{align*}
In addition, assume that for every $\eps>0$ there is $C>0$ such that
for all $N\in\N$ and all bounded Borel functions $F\colon\R\to\C$
with $\supp F\subset[-1,N+1]$
\begin{align}\label{DOS_Thm(3.5)-Hardy}
\bigl\|F(\sqrt[m]L)\bigr\|_{H^1_L(X)\to H^1_L(X)}^2
\leq
CN^{\kappa D+\eps}\|F(N\cdot)\|^2_{N^\kappa,q}\,.
\end{align}
Let $s>\max\{D/2,1/q\}$. Then, for any bounded Borel function
$F\colon\R\to\C$ with
\begin{align*}
\sup_{n\in\N}\|\om F(2^n\cdot)\|_{H_q^s}<\infty\,,
\end{align*}
there exists a constant $C>0$ such that for all $f\in H^1_L(X)$
\begin{align}\label{DOS_Thm(3.6)-Hardyb}
\|F(L)f\|_{H^1_L}
\leq
C\Bigl(\sup_{n\in\N}\|\om
F(2^n\cdot)\|_{H_q^s}+\|F\|_{L^\infty}\Bigr)\|f\|_{H^1_L}\,.
\end{align}
\end{theorem}

The proof of Theorem~\ref{DY_Thm1.1_Plancherel} is essentially 
based on the approach 
in \cite[Theorem~3.2]{DOS}. We omit the details here, and only 
mention that one can use the following $L^2$-version of 
\cite[Lemma~4.3~b)]{DOS} which can be proven in a similar way as 
Lemma~\ref{DOS_Lem4_3a} below.

\begin{lemma}
Let $L$, $\kappa$, $q$ be as in Theorem \ref{DY_Thm1.1_Plancherel}.
For $\xi\in C_c^\infty([-1,1])$ and $N\in\N$ define the function $\xi_N$
via the formula $\xi_N(\la):=N\,\xi(N\la)$. Then for any $s\geq2/q$,
$\eps>0$, and any $\xi\in C_c^\infty([-1,1])$ there exists a
constant $C>0$ such that
\begin{align}\label{DOS_Lem(4.3)Hardy}
&\bigl\|(F\ast\xi_{N^{\kappa-1}})(\sqrt[m]L)\,\cf_{B(y,1/N)}
\bigr\|_{L^2(X)\to L^2(X,(1+Nd(\cdot,y))^sd\mu)}
\leq
C\,\|F(N\cdot)\|_{H_q^{s/2+\eps}}
\end{align}
for all $N\in\N$ with $N>8$, all $y\in X$, and all bounded Borel
functions $F\colon\R\to\C$ with $\supp F\subseteq[N/4,N]$ and
$F(N\cdot)\in H_q^{s/2+\eps}$.
\end{lemma}

The rest of this section is devoted to the proofs of Theorems
\ref{DY_Thm1.1}~a) and \ref{Hardy_Plancherel}. We start with the 
aforementioned statement concerning the validity of the Plancherel 
condition (\ref{DOS_Plancherel_Hardy}) for $q=\infty$.

\begin{lemma}\label{DOS_Lem2_2}
Let $L$ be a non-negative, self-adjoint operator on $L^2(X)$ which
satisfies DG$_m$ for some $m\geq2$. Then there exists a constant
$C>0$ such that for all $R>0$, $y\in X$, and bounded Borel functions
$F\colon[0,\infty)\to\C$ with $\supp F\subset[0,R]$
$$
\bigl\|F(\sqrt[m]L)\,\cf_{B(y,1/R)}\bigr\|_{L^2\to L^2}
\leq
C\,\|F\|_{L^\infty}\,.
$$
\end{lemma}

\begin{proof}
Let $R>0$, $y\in X$ and $F\colon[0,\infty)\to\C$ be a bounded Borel
function whose support is contained in $[0,R]$. For any $\la\geq0$
define $G_1(\la):=F(\sqrt[m]\la)\,e^{\la/R^m}$ and
$G_2(\la):=e^{-\la/R^m}$, so that we can write
$F(\sqrt[m]L)=G_1(L)G_2(L)$. Observe that $\supp G_1\subset[0,R^m]$
and thus $\|G_1(L)\|_{L^2\to L^2}\leq\|G_1\|_{L^\infty}\leq
e\,\|F\|_{L^\infty}$. As $L$ fulfills DG$_m$, we deduce with the
help of Fact \ref{GGEequiv} that
\begin{align*}
\bigl\|G_2(L)\,\cf_{B(y,1/R)}\bigr\|_{L^2\to L^2}
&\leq
\sum_{k=0}^\infty\,\bigl\|\cf_{A(y,1/R,k)}e^{-\frac1{R^m}\,L}\cf_{B(y,1/R)}
\bigr\|_{L^2\to L^2}
\leqC
\sum_{k=0}^\infty e^{-bk^{\frac{m}{m-1}}}
\leqC
1\,.
\end{align*}
Combining these estimates gives the desired bound
\begin{align*}
\bigl\|F(\sqrt[m]L)\,\cf_{B(y,1/R)}\bigr\|_{L^2\to L^2}
&\leq
\|G_1(L)\|_{L^2\to L^2}\,\bigl\|G_2(L)\,\cf_{B(y,1/R)}\bigr\|_{L^2\to L^2}
\leqC
\|F\|_{L^\infty}
\,.
\end{align*}
\end{proof}

Now, we provide a criterion for the boundedness of spectral
multipliers on the Hardy space $H^1_L(X)$. Our result, presented in
Theorem \ref{DY_Thm3.1} below, generalizes the statement
\cite[Theorem 3.1]{DY} due to X.T.\ Duong and L.X.\ Yan which merely
works under Davies-Gaffney estimates of order $m=2$. Afterwards we
check that the assumption (\ref{DY(3.1)}) holds whenever the
involved function $F$ satisfies the assumptions of one of the above
theorems.

\begin{theorem}\label{DY_Thm3.1}
Let $L$ be an injective, non-negative, self-adjoint operator on
$L^2(X)$ which satisfies Davies-Gaffney estimates DG$_m$ for some
$m\geq2$. Further, let $F\colon[0,\infty)\to\C$ be a bounded Borel
function. Assume that there exist an integer $M>{D}/{m}$ and
constants $C_F>0$, $\delta>D/2$ such that
\begin{align}\label{DY(3.1)}
\bigl\|\cf_{U_j(B)}F(L)(I-e^{-r^mL})^M\cf_B\bigr\|_{L^2\to L^2}
\leq
C_F\,2^{-j\delta}
\end{align}
for every $j\in\N\setminus\{1\}$ and every ball $B\subset X$ with
radius $r$. As usual, $U_j(B)$ stands for the dyadic annular set as
defined in (\ref{annuli}). Then the operator $F(L)$ extends from
$H_L^1(X)\cap L^2(X)$ to a bounded linear operator on $H_L^1(X)$.
More precisely, there exists a constant $C>0$ such that
\begin{align*}
\|F(L)\|_{H_L^1(X)\to H_L^1(X)}\leq CC_F\,.
\end{align*}
\end{theorem}

The strategy of proof consists in reducing the statement to uniform
boundedness of the $H_L^1(X)$-norm of $F(L)a$ for every
$(2M,\wt\eps,L)$-molecule $a$. Recall that $a$ can be rewritten as
$a=L^{2M}b$ for some $b\in\dom(L^{2M})$. By the lack of information
on the support of $L^kb$ for $k\in\{0,1,\ldots,2M\}$, we cannot
apply (\ref{DY(3.1)}) directly. Instead we shall choose $\wt\eps$
large enough and use an estimate of annular type furnished by the
next lemma whose proof is postponed to Section~\ref{proofs}.

\begin{lemma}\label{LemDY(3.1)VarianteLem}
Suppose that the operator $L$ and the function $F$ have the same
properties as in Theorem \ref{DY_Thm3.1}. Then there exists a
constant $C>0$ such that
\begin{align}\label{DY(3.1)Variante}
\bigl\|\cf_{U_j(B)}F(L)(I-e^{-r^mL})^M\cf_{U_i(B)}\bigr\|_{L^2\to L^2}
\leq
CC_F\,2^{iD}\,2^{-|j-i|\delta}
\end{align}
for every $i,j\in\N\setminus\{1\}$ and every ball $B\subset X$
with radius $r$.
\end{lemma}

Next, we provide the technical result that an integrated version of
the regularization operator $(I-e^{-r^mL})^M$ satisfies
$L^2(X)$-norm estimates of annular type if $L$ fulfills DG$_m$. This
will be achieved with a similar reasoning as in the proof of the
preceding statement (cf.\ Section \ref{proofs}).

\begin{lemma}\label{intVerReg}
Let $K\in\N$ and $L$ be an injective, non-negative, self-adjoint
operator on $L^2(X)$ which fulfills Davies-Gaffney estimates DG$_m$
for some $m\geq2$. For $M\in\N$ and $r>0$ define the operator
\begin{align}\label{ewrfjhgef}
P_{m,M,r}(L):=r^{-m}\int_r^{\sqrt[m]2r}s^{m-1}(I-e^{-s^mL})^M\,ds\,.
\end{align}
Then there exist $b,C>0$ such that for any $i,j\in\N_0$ and
arbitrary balls $B\subset X$ of radius $r$
\begin{align}\label{intVerRegDG}
\bigl\|\cf_{U_j(B)}P_{m,M,r}(L)^{K}\cf_{U_i(B)}\bigr\|_{L^2\to L^2}
&\leq
C\exp\bigl(-b\,2^{|j-i|}\bigr)\,.
\end{align}
Here, the constants $b,C$ depend exclusively on $m,K,M$ and the
constants appearing in the Davies-Gaffney and doubling condition.
\end{lemma}

With the preceding lemmas at hand, we are prepared for the proof of
Theorem~\ref{DY_Thm3.1}. Here, we rely to a large extent on the
proof of \cite[Theorem 3.1]{DY}.

\begin{spezbew}{Theorem \ref{DY_Thm3.1}}
Let $F\colon[0,\infty)\to\C$ be a bounded Borel function such that
(\ref{DY(3.1)}) holds for some constants $C_F>0$, $\del>D/2$, and
$M\in\N$ with $M>D/m$.

First of all, we note that the operator $F(L)$ can be defined via
(\ref{specthmL2}) on the set $H_L^1(X)\cap L^2(X)$ which is dense 
in $H_L^1(X)$ (cf.\ Definition~\ref{DefHardy}).

Let $\wt\del\in(D/2,\min\{\del,mM-D/2\})$ be fixed. Define
$\eps:=\wt\delta-D/2>0$ and $\wt\eps:=D+\wt\del$. We claim that, for
every $(2M,\wt\eps,L)$-molecule $a$, $F(L)a$ is, up to
multiplication by a constant independent of $a$, an
$(M,\eps,L)$-molecule. The conclusion of Theorem \ref{DY_Thm3.1} is
then an immediate consequence of Corollary~\ref{MolH1}. Indeed, by
Theorem~\ref{sqfct=mol}, every $f\in H_L^1(X)\cap L^2(X)$ admits a
molecular $(2M,\wt\eps,L)$-represen\-tation, i.e.\ there exist a
scalar sequence $(\la_j)_{j\in\N_0}\in\ell^1$ and a sequence
$(m_j)_{j\in\N_0}$ of $(2M,\wt\eps,L)$-molecules such that
$$
f=\sum_{j=0}^\infty\la_jm_j
$$
in $L^2(X)$ and
$$
\|f\|_{H_L^1}\cong\sum_{j=0}^\infty|\la_j|
$$
with implicit constants independent of $f$. Therefore, we have
$$
\|F(L)f\|_{H^1_L}
\leq
\sum_{j=0}^\infty|\la_j|\,\|F(L)m_j\|_{H^1_L}\,.
$$
But by the claim above, $F(L)m_j$ is a constant multiple of an
$(M,\eps,L)$-molecule. Hence, we conclude from Corollary~\ref{MolH1}
that the $H^1_L(X)$-norm of $F(L)m_j$ is bounded by a constant $C>0$
being independent of $j$. Thus, once the above claim is proved, the
boundedness of $F(L)$ on $H_L^1(X)$ is shown since
$$
\|F(L)f\|_{H^1_L}
\leq
\sum_{j=0}^\infty|\la_j|\,\|F(L)m_j\|_{H^1_L}
\leq
C\sum_{j=0}^\infty|\la_j|
\cong
\|f\|_{H^1_L}
$$
for any $f\in H^1_L(X)\cap L^2(X)$ and $H^1_L(X)\cap L^2(X)$ is
dense in the Hardy space $H^1_L(X)$.

\smallskip
Now we proceed with the proof of the claim stated above. Let $a$ be
an $(2M,\wt\eps,L)$-molecule. According to Definition
\ref{molecule}, we find a function $b\in\dom(L^{2M})$ and a ball
$B\subset X$ such that $a=L^{2M}b$ and (\ref{mol}) hold. By the
spectral theorem for $L$, we may write
$$
F(L)a=L^M\bigl(F(L)L^Mb\bigr)\,.
$$
In particular, $F(L)L^Mb$ belongs to $\dom(L^M)$. For the proof that
$F(L)a$ is a constant multiple of an $(M,\eps,L)$-molecule it
remains to check (ii) from Definition \ref{molecule}, i.e.\ the
existence of a constant $C>0$ such that for all $j\in\N_0$ and all
$k\in\{0,1,\ldots,M\}$
\begin{align}\label{DY(3.4)}
\bigl\|(r^mL)^k\bigl(F(L)L^{M}b\bigr)\bigr\|_{L^2(U_j(B))}
\leq
CC_Fr^{mM}2^{-j\eps}\mu(2^jB)^{-1/2}\,,
\end{align}
where $r$ denotes the radius of the ball $B$.

For $j\in\{0,1,2\}$, we employ the boundedness of $F(L)$ on $L^2(X)$
as well as the properties of the $(2M,\wt\eps,L)$-molecule $a$. For
any $k\in\{0,1,\ldots,M\}$, this leads to
\begin{align}\label{DY(3.34)}
\bigl\|(r^mL)^k\bigl(F(L)L^Mb\bigr)\bigr\|_{L^2(U_j(B))}
&\leq
r^{mk}\|F(L)\|_{L^2\to L^2}\,r^{-m(M+k)}\|(r^mL)^{M+k}b\|_{L^2}\nonumber
\\&\leq
\|F\|_{L^\infty}\,r^{-mM}\sum_{i=0}^\infty\|(r^mL)^{M+k}b\|_{L^2(U_i(B))}\nonumber
\\&\leq
\|F\|_{L^\infty}\,r^{-mM}\sum_{i=0}^\infty
r^{2mM}2^{-i\wt\eps}\mu(2^iB)^{-1/2}\nonumber
\\&\leqC\nonumber
\|F\|_{L^\infty}\,r^{mM}\mu(B)^{-1/2}
\\&\leqC
\bigl(\|F\|_{L^\infty}2^{D+2\eps}\bigr)\,r^{mM}2^{-j\eps}\mu(2^jB)^{-1/2}\,.
\end{align}

Now assume that $j\geq3$. We start by representing the identity on
$L^2(X)$ with the help of the operators $e^{-\nu r^mL}$ and
$P_{m,M,r}(L)$, where the latter has been defined in
(\ref{ewrfjhgef}). Applying this to $(r^mL)^k(F(L)L^Mb)$, the
procedure produces a regularizing effect for the operator $F(L)$ and
finally permits us to insert the assumption (\ref{DY(3.1)}) in the
version of Lemma \ref{LemDY(3.1)VarianteLem} and the Davies-Gaffney
estimates in the form of Lemma \ref{intVerReg}. Inspired by
\cite[(8.7), (8.8)]{HM}, we use the elementary equations
\begin{align*}
1=mr^{-m}\int_r^{\sqrt[m]2r}s^{m-1}\,ds
\end{align*}
and
$$
1=(1-e^{-s^m\la})^M-\sum_{\nu=1}^M\binom M\nu(-1)^\nu e^{-\nu
s^m\la}\qquad(\la\geq0,s>0)
$$
to deduce, via the spectral theorem for $L$,
\begin{align}\label{DY(3.5)}
I
=
mr^{-m}\int_r^{\sqrt[m]2r}s^{m-1}(I-e^{-s^mL})^M\,ds
+
\sum_{\nu=1}^M\nu C_{\nu,M}mr^{-m}\int_r^{\sqrt[m]2r}s^{m-1}e^{-\nu s^mL}\,ds\,,
\end{align}
where $C_{\nu,M}:=\frac{(-1)^{\nu+1}}\nu\binom M\nu$. Further, we
have $\partial_se^{-\nu s^mL}=-\nu ms^{m-1}Le^{-\nu s^mL}$ and
therefore
\begin{align}\label{DY(3.6)}
\nu mL\int_r^{\sqrt[m]2r}s^{m-1}e^{-\nu s^mL}\,ds\nonumber
&=
e^{-\nu r^mL}-e^{-2\nu r^mL}
=
e^{-\nu r^mL}(I-e^{-\nu r^mL})
\\&=
e^{-\nu r^mL}(I-e^{-r^mL})\sum_{\eta=0}^{\nu-1}e^{-\eta r^mL}\,.
\end{align}
By recalling the definition of $P_{m,M,r}(L)$ and by inserting the
equation (\ref{DY(3.6)}) into (\ref{DY(3.5)}), we end up with the
following formula for the identity on $L^2(X)$
\begin{align*}
I
=
mP_{m,M,r}(L)
+
\sum_{\nu=1}^M
C_{\nu,M}r^{-m}L^{-1}(I-e^{-r^mL})\sum_{\eta=\nu}^{2\nu-1}e^{-\eta r^mL}\,.
\end{align*}
Expanding the identity $I^M$ by means of the binomial formula leads
to
\begin{align*}
I&=
\bigl(mP_{m,M,r}(L)\bigr)^M
\\&\quad+
\sum_{l=1}^M\binom Ml\Biggl(\sum_{\nu=1}^MC_{\nu,M}r^{-m}L^{-1}(I-e^{-r^mL})
\sum_{\eta=\nu}^{2\nu-1}e^{-\eta r^mL}\Biggl)^{l}\bigl(mP_{m,M,r}(L)\bigr)^{M-l}
\\&=
m^MP_{m,M,r}(L)^M
+
\sum_{l=1}^Mr^{-ml}L^{-l}(I-e^{-r^mL})^l P_{m,M,r}(L)^{M-l}
\sum_{\nu=1}^{(2M-1)l}C_{l,\nu,m,M}e^{-\nu r^mL}
\end{align*}
for appropriate constants $C_{l,\nu,m,M}$ depending on the
subscripted parameters.

\smallskip
Now fix $k\in\{0,1,\ldots,M\}$. The above identity allows us to
represent $(r^mL)^k(F(L)L^Mb)$ in the following way
\begin{align*}
&(r^mL)^k\bigl(F(L)L^Mb\bigr)=
m^Mr^{mk}P_{m,M,r}(L)^MF(L)(L^{M+k}b)
\\&\qquad+
\sum_{l=1}^Mr^{mk-ml}L^{-l}(I-e^{-r^mL})^l P_{m,M,r}(L)^{M-l}
\sum_{\nu=1}^{(2M-1)l}C_{l,\nu,m,M}e^{-\nu r^mL}F(L)(L^{M+k}b)
\\&\qquad=:
\sum_{l=0}^MG^{(k)}_{l,M,r}\,.
\end{align*}

We shall establish an adequate bound on
$\|G^{(k)}_{l,M,r}\|_{L^2(U_j(B))}$ by distinguishing the three
cases $l=0$, $l\in\{1,\ldots,M-1\}$, and $l=M$.

\smallskip
\noindent{\em Case 1:} $l=0$.

\noindent
First, we write for $\mu$-a.e.\ $x\in X$
\begin{align*}
&\bigl|G^{(k)}_{0,M,r}(x)\bigr|
=
m^Mr^{mk}\bigl|P_{m,M,r}(L)\bigl(P_{m,M,r}(L)^{M-1}F(L)(L^{M+k}b)\bigr)(x)\bigr|
\\&\quad\leq
m^Mr^{mk}r^{-m}\int_r^{\sqrt[m]2r}s^{m-1}\Bigl|P_{m,M,r}(L)^{M-1}
\bigl(F(L)(I-e^{-s^mL})^M(L^{M+k}b)\bigr)(x)\Bigr|\,ds
\\&\quad\leq
\sum_{i=0}^\infty m^Mr^{mk}r^{-m}
~\times\\&\quad\qquad\times~
\int_r^{\sqrt[m]2r}s^{m-1}\Bigl|P_{m,M,r}(L)^{M-1}
\Bigl(\cf_{U_i(B)}\bigl(F(L)(I-e^{-s^mL})^M(L^{M+k}b)\bigr)\Bigr)(x)\Bigr|\,ds\,
.
\end{align*}
As seen in Lemma \ref{intVerReg}, the operator $P_{m,M,r}(L)^{M-1}$
enjoys the off-diagonal estimate (\ref{intVerRegDG}). This yields
\begin{align*}
&\bigl\|G^{(k)}_{0,M,r}\bigr\|_{L^2(U_j(B))}
\\&\leq
m^Mr^{mk}
\sum_{i=0}^\infty r^{-m}~\times
\\&\qquad\times~
\int_r^{\sqrt[m]2r}s^{m-1}\Bigl\|P_{m,M,r}(L)^{M-1}\Bigl(\cf_{U_i(B)}
\bigl(F(L)(I-e^{-s^mL})^M(L^{M+k}b)\bigr)\Bigr)\Bigr\|_{L^2(U_j(B))}\,ds
\\&\leqC
r^{mk}
\sum_{i=0}^\infty\exp\bigl(-b\,2^{|j-i|}\bigr)\,
r^{-m}\int_r^{\sqrt[m]2r}s^{m-1}
\bigl\|F(L)(I-e^{-s^mL})^M(L^{M+k}b)\bigr\|_{L^2(U_i(B))}\,ds\,.
\end{align*}
In order to apply Lemma \ref{LemDY(3.1)VarianteLem}, we first
observe that for every $s\in[r,\sqrt[m]2r]$ the ball $U_0(B)$ is
contained in $U_0(B(x_B,s))$ and the annulus $U_i(B)$ in
$U_{i-1}(B(x_B,s))\cup U_i(B(x_B,s))$ for each $i\in\N$ where $x_B$
denotes the center of $B$. These inclusions give for every
$s\in[r,\sqrt[m]2r]$
\begin{align}\label{rwflrgklfwebfwklf}
&\bigl\|F(L)(I-e^{-s^mL})^M(L^{M+k}b)\bigr\|_{L^2(U_i(B))}\nonumber
\\&\quad\leq\nonumber
\sum_{\nu=i-1}^{i}\bigl\|F(L)(I-e^{-s^mL})^M(L^{M+k}b)\bigr\|_{L^2(U_\nu(B(x_B,
s)))}
\\&\quad\leq\nonumber
\sum_{\nu=i-1}^{i}\biggl(\bigl\|F(L)(I-e^{-s^mL})^M(
\cf_{B(x_B,s)}L^{M+k}b)\bigr\|_{L^2(U_\nu(B(x_B,s)))}
\\&\quad\qquad\qquad+\sum_{\eta=1}^\infty
\bigl\|F(L)(I-e^{-s^mL})^M(\cf_{U_\eta(B(x_B,s))}L^{M+k}b)\bigr\|_{
L^2(U_\nu(B(x_B,s)))}\biggr)\,.
\end{align}
Due to (\ref{DY(3.1)}), the first summand in the bracket is bounded
by
\begin{align*}
C_F2^{-\nu\del}\|L^{M+k}b\|_{L^2(B(x_B,s))}
\leq
C_F2^{-\nu\del}\bigl(\|L^{M+k}b\|_{L^2(B)}+\|L^{M+k}b\|_{L^2(U_1(B))}\bigr)\,.
\end{align*}
By recalling the properties of the $(2M,\wt\eps,L)$-molecule $a$, we
obtain
\begin{align*}
\|L^{M+k}b\|_{L^2(B)}
&=
r^{-m(M+k)}\|(r^mL)^{M+k}b\|_{L^2(B)}
\\&
\leq
r^{-m(M+k)}r^{2mM}\mu(B)^{-1/2}
=
r^{mM-mk}\mu(B)^{-1/2}
\end{align*}
as well as
\begin{align*}
\|L^{M+k}b\|_{L^2(U_1(B))}
&=
r^{-m(M+k)}\|(r^mL)^{M+k}b\|_{L^2(U_1(B))}
\\&\leq
r^{-m(M+k)}r^{2mM}2^{-\wt\eps}\mu(2B)^{-1/2}
\leq
r^{mM-mk}\mu(B)^{-1/2}\,.
\end{align*}
Hence, we have the bound
\begin{align}\label{rwkftkwef} 
\bigl\|F(L)(I-e^{-s^mL})^M(\cf_{B(x_B,s)}L^{M+k}b)\bigr\|_{L^2(U_\nu(B(x_B,s)))}
\leqC
C_Fr^{mM-mk}2^{-\nu\del}\mu(B)^{-1/2}\,.
\end{align}
The series in the bracket of (\ref{rwflrgklfwebfwklf}) can be
estimated with the help of Lemma \ref{LemDY(3.1)VarianteLem}
\begin{align*}
&\sum_{\eta=1}^\infty
\bigl\|F(L)(I-e^{-s^mL})^M(\cf_{U_\eta(B(x_B,s))}L^{M+k}b)\bigr\|_{
L^2(U_\nu(B(x_B,s)))}
\\&\quad\leqC
\sum_{\eta=1}^\infty C_F2^{\eta
D}2^{-|\nu-\eta|\del}\|L^{M+k}b\|_{L^2(U_\eta(B(x_B,s)))}\,.
\end{align*}
Since $a$ is an $(2M,\wt\eps,L)$-molecule, we obtain
\begin{align*}
&\|L^{M+k}b\|_{L^2(U_\eta(B(x_B,s)))}
\\&\quad\leq
r^{-m(M+k)}\bigl(\|(r^mL)^{M+k}b\|_{L^2(U_\eta(B(x_B,r)))}+
\|(r^mL)^{M+k}b\|_{L^2(U_{\eta+1}(B(x_B,r)))}\bigr)
\\&\quad\leq
r^{-m(M+k)}\bigl(r^{2mM}2^{-\eta\wt\eps}\mu(2^\eta B(x_B,r))^{-1/2}+
r^{2mM}2^{-(\eta+1)\wt\eps}\mu(2^{\eta+1}B(x_B,r))^{-1/2}\bigr)
\\&\quad\leqC
r^{mM-mk}2^{-\eta\wt\eps}\mu(B(x_B,r))^{-1/2}
\end{align*}
and thus
\begin{align}\label{wrfkjfwkbfwfln}
&\sum_{\eta=1}^\infty\nonumber
\bigl\|F(L)(I-e^{-s^mL})^M(\cf_{U_\eta(B(x_B,s))}L^{M+k}b)\bigr\|_{
L^2(U_\nu(B(x_B,s)))}
\\&\quad\leqC\nonumber
C_Fr^{mM-mk}\mu(B(x_B,r))^{-1/2}\sum_{\eta=1}^\infty
2^{-\eta(\wt\eps-D)}2^{-|\nu-\eta|\del}
\\&\quad\leqC
C_Fr^{mM-mk}2^{-\nu\wt\delta}\mu(B(x_B,r))^{-1/2}\,.
\end{align}
In the last step we used the fact that
\begin{align*}
\sum_{\eta=1}^\infty2^{-\eta(\wt\eps-D)}2^{-|\nu-\eta|\del}
&=
2^{-\nu(\wt\eps-D)}\Biggl(\sum_{n=-\infty}^{0}2^{n(\wt\eps-D)}2^{-|n|\delta}
+
\sum_{n=1}^{\nu-1}2^{n(\wt\eps-D)}2^{-n\delta}\Biggr)
\\&\leq
2^{-\nu\wt\delta}\Biggl(\sum_{n=-\infty}^{0}2^{-|n|\delta}
+
\sum_{n=1}^\infty2^{-n(D+\delta-\wt\eps)}\Biggr)
\leqC
2^{-\nu\wt\delta}\,.
\end{align*}
In view of the inequalities (\ref{rwkftkwef}) and
(\ref{wrfkjfwkbfwfln}), we have the following estimate of
(\ref{rwflrgklfwebfwklf})
\begin{align*}
&\bigl\|F(L)(I-e^{-s^mL})^M(L^{M+k}b)\bigr\|_{L^2(U_i(B))}
\leqC C_Fr^{mM-mk}2^{-i\wt\del}\mu(B)^{-1/2}\,.
\end{align*}
With the help of this bound and the doubling property,
we continue
\begin{align}\label{DY(3.9)}
&\bigl\|G^{(k)}_{0,M,r}\bigr\|_{L^2(U_j(B))}\nonumber
\\&\quad\leqC\nonumber
r^{mk}\sum_{i=0}^\infty
\exp\bigl(-b\,2^{|j-i|}\bigr)\,
r^{-m}\int_r^{\sqrt[m]2r}s^{m-1}
\bigl\|F(L)(I-e^{-s^mL})^M(L^{M+k}b)\bigr\|_{L^2(U_i(B))}\,ds
\\&\quad\leqC\nonumber
r^{mk}
\sum_{i=0}^\infty
\exp\bigl(-b\,2^{|j-i|}\bigr)\,
r^{-m}\int_r^{\sqrt[m]2r}s^{m-1}\,ds\,\,
C_Fr^{mM-mk}2^{-i\wt\delta}\mu(B)^{-1/2}
\\&\quad\leqC
C_Fr^{mM}2^{-j\wt\delta}\mu(B)^{-1/2}
\leqC
C_Fr^{mM}2^{-j(\wt\delta-D/2)}\mu(2^jB)^{-1/2}\,.
\end{align}
In the second to the last step we used, among other things, the
following fact which is easily verified by an index shift
\begin{align}\label{rwlkfreflkregflkbe}
\sum_{i=0}^\infty
\exp\bigl(-b\,2^{|j-i|}\bigr)\,2^{-i\wt\delta}\nonumber
&=\nonumber
\sum_{n=-\infty}^0
\exp\bigl(-b\,2^{|n|}\bigr)\,2^{-(j-n)\wt\delta}
+
\sum_{n=1}^j
\exp\bigl(-b\,2^{|n|}\bigr)\,2^{-(j-n)\wt\delta}
\\&\quad\leq
2^{-j\wt\delta}\sum_{n=-\infty}^\infty
\exp\bigl(-b\,2^{|n|}\bigr)\,2^{n\wt\delta}
\leqC
2^{-j\wt\delta}\,.
\end{align}

\smallskip
\noindent{\em Case 2:} $l\in\{1,2,\ldots,M-1\}$.

\noindent We have for $\mu$-a.e.\ $x\in X$
\begin{align*}
\bigl|G^{(k)}_{l,M,r}(x)\bigr|&
\leq
r^{m(k-l)}\sum_{\nu=1}^{(2M-1)l}|C_{l,\nu,m,M}|
\int_r^{\sqrt[m]2r}\Bigl(\frac sr\Bigr)^m
\Bigl|L^{M-l}e^{-\nu r^mL}(I-e^{-r^mL})^l
~\circ\\&\qquad
\circ~P_{m,M,r}(L)^{M-l-1}\bigl(F(L)(I-e^{-s^mL})^M(L^{k}b)\bigr)(x)\Bigr|\,
\frac{ds}s
\\&\leqC
r^{m(k-M)}\sum_{\nu=1}^{(2M-1)l}
\sum_{i=0}^\infty
\int_r^{\sqrt[m]2r}\Bigl|
(r^mL)^{M-l}
e^{-\nu r^mL}(I-e^{-r^mL})^l
~\circ\\&\qquad
\circ~P_{m,M,r}(L)^{M-l-1}\Bigl(\cf_{U_i(B)}\bigl(F(L)(I-e^{-s^mL})^M(L^{k}
b)\bigr)
\Bigr)(x)\Bigr|\,\frac{ds}s\,.
\end{align*}
By Lemma \ref{DLProp2.3}, the operator family $\{(tL)^{M-l}e^{-\nu
tL}:\,t>0\}$ satisfies DG$_m$. After writing $(I-e^{-tL})^l$ with
the help of the binomial formula, it is straightforward to prove
that DG$_m$ also holds for $\{(tL)^{M-l}e^{-\nu
tL}(I-e^{-tL})^l:\,t>0\}$. Hence, one can show $L^2(X)$-norm
estimates of annular type similar to those in (\ref{lrfrqrf}) below 
for operators of the form $(r^mL)^{M-l}e^{-\nu r^mL}(I-e^{-r^mL})^l$
whenever $r$ denotes the radius of the considered ball. Thanks to
Lemma \ref{intVerReg}, $P_{m,M,r}(L)^{M-l-1}$ fulfills
(\ref{intVerRegDG}). If one adapts the arguments given at the end of
the proof of Lemma \ref{intVerReg} (cf.\ Section \ref{proofs}), one
can verify that the composition of these operators enjoys the
following version of (\ref{intVerRegDG})
\begin{align*}
\bigl\|\cf_{U_j(B)}
(r^mL)^{M-l}e^{-\nu r^mL}(I-e^{-r^mL})^lP_{m,M,r}(L)^{M-l-1}
\cf_{U_i(B)}\bigr\|_{2\to2}
&\leq
C \exp\bigl(-b\,2^{|j-i|}\bigr)
\end{align*}
for some constants $b,C>0$ depending only on $m,K,M$ and the
constants in the Davies-Gaffney and the doubling condition.

This estimate leads to
\begin{align}\label{case1.2}
\bigl\|G^{(k)}_{l,M,r}\bigr\|_{L^2(U_j(B))}\nonumber
&\leqC
r^{m(k-M)}\sum_{\nu=1}^{(2M-1)l}\sum_{i=0}^\infty
\exp\bigl(-b\,2^{|j-i|}\bigr)
~\times\\&\qquad\times~
\int_r^{\sqrt[m]2r}
\bigl\|F(L)(I-e^{-s^mL})^M(L^{k}b)\bigr\|_{L^2(U_i(B))}\frac{ds}s\,.
\end{align}
By employing similar arguments as in Case 1 (just replace $L^{M+k}b$
by $L^{k}b$), we conclude that for any $i\in\N_0$ and
$s\in[r,\sqrt[m]2r]$
\begin{align}\label{erklngea}
\bigl\|F(L)(I-e^{-s^mL})^M(L^{k}b)\bigr\|_{L^2(U_i(B))}
&\leqC
C_Fr^{2mM-mk}2^{-i\wt\del}\mu(B)^{-1/2}\,.
\end{align}
Inserting this estimate into (\ref{case1.2}) yields readily
\begin{align}\label{DYcase1.3}
\bigl\|G^{(k)}_{l,M,r}\bigr\|_{L^2(U_j(B))}
&\leqC
C_Fr^{mM}2^{-j(\wt\delta-D/2)}\mu(2^jB)^{-1/2}\,.
\end{align}

\smallskip
\noindent{\em Case 3:} $l=M$.

\noindent
In this case we have
\begin{align*}
G^{(k)}_{M,M,r}
&=
r^{m(k-M)}\sum_{\nu=1}^{(2M-1)M}C_{M,\nu,m,M}
e^{-\nu r^mL}\bigl(F(L)(I-e^{-r^mL})^M(L^{k}b)\bigr)
\\&=
r^{m(k-M)}\sum_{\nu=1}^{(2M-1)M}C_{M,\nu,m,M}
\sum_{i=0}^\infty e^{-\nu
r^mL}\Bigl(\cf_{U_i(B)}\bigl(F(L)(I-e^{-r^mL})^M(L^{k}b)\bigl)\Bigl)\,.
\end{align*}
With the help of (\ref{lrfrqrf}), (\ref{rgfrejfhlreb}) below,
and (\ref{erklngea}), (\ref{rwlkfreflkregflkbe}), we obtain
\begin{align*}
\bigl\|G^{(k)}_{M,M,r}\bigr\|_{L^2(U_j(B))}
&\leqC
r^{m(k-M)}\sum_{i=0}^\infty
\exp\bigl(-b\,2^{|j-i|}\bigr)
\bigl\|F(L)(I-e^{-r^mL})^M(L^{k}b)\bigr\|_{L^2(U_i(B))}
\\&\leqC
C_Fr^{mM}\sum_{i=0}^\infty\exp\bigl(-b\,2^{|j-i|}\bigr)\,2^{-i\wt\delta}\mu(B)^{
-1/2}
\\&\leqC
C_Fr^{mM}2^{-j\wt\delta}\mu(B)^{-1/2}
\\&\leqC
C_Fr^{mM}2^{-j(\wt\del-D/2)}\mu(2^jB)^{-1/2}\,.
\end{align*}

\smallskip
This, in combination with (\ref{DY(3.34)}), (\ref{DY(3.9)}), and
(\ref{DYcase1.3}), gives the desired estimate (\ref{DY(3.4)}).
\end{spezbew}

We prepare the proof
of Theorem \ref{DY_Thm1.1}~a) with the next two lemmas. The first
one corresponds to \cite[Lemma 4.1]{DOS} and gives an extension of
generalized Gaussian estimates from real times to complex times in
some weighted space. This is crucial for our proof of Lemma
\ref{DOS_Lem4_3a}, where the operator $F(\sqrt[m]L)$ will be
represented in terms of the extended semigroup $(e^{-zL})_{\Re z>0}$
by a Fourier transform argument taken from \cite{DOS}.

\begin{lemma}\label{DOS_Lem4_1}
Let $s\geq0$. In the situation of Theorem \ref{DY_Thm1.1}~a), there
exists a constant $C>0$ such that for all $R>0$, $\tau\in\R$, and
$y\in X$
$$
\bigl\|e^{-(1+i\tau)R^{-m}L}\cf_{B(y,1/R)}\bigr\|_{L^2(X)\to
L^2(X,(1+Rd(\cdot,y))^sd\mu)}
\leq C
(1+\tau^2)^{s/4}\,.
$$
\end{lemma}

\begin{proof}
According to Fact \ref{GGEkomplex}, there are constants $b,C>0$ such
that for all $x,y\in X$ and all $z\in\C$ with $\Re z>0$
$$
\bigl\|\cf_{B(x,r_z)}e^{-zL}\cf_{B(y,r_z)}\bigr\|_{L^2\to L^2}
\leq
C\exp\Biggl(-b\biggl(\frac{d(x,y)}{r_z}\biggr)^{\frac{m}{m-1}}\Biggr)\,,
$$
where $r_z:=(\Re z)^{1/m-1}|z|$. By Fact \ref{GGEequiv}, this
two-ball estimate is equivalent to the assertion that there exist
$b,C>0$ such that for every $k\in\N_0$, $y\in X$, and $z\in\C$ with
$\Re z>0$
$$
\bigl\|\cf_{A(y,r_z,k)}e^{-\ov zL}\cf_{B(y,r_z)}\bigr\|_{L^2\to L^2}
\leq
C\exp\bigl(-bk^{\frac{m}{m-1}}\bigr)\,.
$$
Now let $R>0$, $s\geq0$, $\tau\in\R$, and $y\in X$ be fixed. For
$z:=(1+i\tau)R^{-m}$ we calculate $r_z=(1+\tau^2)^{1/2}/R\geq1/R$
and obtain
\begin{align*}
&\bigl\|e^{-(1+i\tau)R^{-m}L}\cf_{B(y,1/R)}\bigr\|_{L^2(X)\to
L^2(X,(1+Rd(\cdot,y))^sd\mu)}
\\&\quad\leq
\sum_{k=0}^\infty\,\bigl\|\cf_{A(y,r_z,k)}e^{-(1+i\tau)R^{-m}L}\cf_{B(y,1/R)}
\bigr\|_{L^2(X)\to L^2(X,(1+Rd(\cdot,y))^sd\mu)}
\\&\quad\leq
\sum_{k=0}^\infty\,\bigl(1+R(k+1)r_z\bigr)^{s/2}\,
\bigl\|\cf_{A(y,r_z,k)}e^{-(1+i\tau)R^{-m}L}\cf_{B(y,1/R)}\bigr\|_{L^2\to L^2}
\\&\quad\leq
\sum_{k=0}^\infty\,\bigl(1+(k+1)(1+\tau^2)^{1/2}\bigr)^{s/2}\,
\bigl\|\cf_{A(y,r_z,k)}e^{-(1+i\tau)R^{-m}L}\cf_{B(y,r_z)}\bigr\|_{L^2\to L^2}
\\&\quad\leq C
(1+\tau^2)^{s/4}\sum_{k=0}^\infty\,(k+2)^{s/2}\exp\bigl(-bk^{\frac{m}{m-1}}
\bigr)
\\&\quad\leqC
(1+\tau^2)^{s/4}\,.
\end{align*}
\end{proof}

The second preparatory statement is based on \cite[Lemma 4.3
a)]{DOS} and is used to transfer regularity of a function $F$ to an
off-diagonal $L^2$-estimate for $F(\sqrt[m]L)$. The only difference
between (\ref{DOS_Lem(4.2)_o_Planch}) and (\ref{DOS_Lem(4.2)}) is in
the norm of $F(R\cdot)$.

\begin{lemma}\label{DOS_Lem4_3aDG} \label{DOS_Lem4_3a}
Let $L$ be a non-negative, self-adjoint operator on $L^2(X)$
satisfying Davies-Gaffney estimates DG$_{m}$ for some $m\geq2$.
\begin{enumerate}
\item[\bf a)]
Then for any $s\geq0$ and $\eps>0$ there exists a constant $C>0$
such that
\begin{align}\label{DOS_Lem(4.2)_o_Planch}
\bigl\|F(\sqrt[m]L)\,\cf_{B(y,1/R)}\bigr\|_{L^2(X)\to
L^2(X,(1+Rd(\cdot,y))^sd\mu)}
\leq
C\,\|F(R\cdot)\|_{H_2^{(s+1)/2+\eps}}
\end{align}
for every $R>0$, $y\in X$, and every bounded Borel function
$F\colon[0,\infty)\to\C$ with $\supp F\subset[R/4,R]$ and
$F(R\cdot)\in H_2^{(s+1)/2+\eps}$.
\item[\bf b)]
Suppose additionally that $L$ fulfills the Plancherel condition
(\ref{DOS_Plancherel_Hardy}) for some $q\in[2,\infty]$. Then for
any $s\geq2/q$ and $\eps>0$ there exists a constant $C>0$ such that
\begin{align}\label{DOS_Lem(4.2)}
\bigl\|F(\sqrt[m]L)\,\cf_{B(y,1/R)}\bigr\|_{L^2(X)\to
L^2(X,(1+Rd(\cdot,y))^sd\mu)}
\leq
C\,\|F(R\cdot)\|_{H_q^{s/2+\eps}}
\end{align}
for every $R>0$, $y\in X$, and every bounded Borel function
$F\colon[0,\infty)\to\C$ with $\supp F\subset[R/4,R]$ and
$F(R\cdot)\in H_q^{s/2+\eps}$.
\end{enumerate}
\end{lemma}

\begin{proof}
Let $R>0$ and $F\colon[0,\infty)\to\C$ be a bounded Borel function
with $\supp F\subset[R/4,R]$. For all $\la\geq0$ define
$G(\la):=F(R\sqrt[m]\la)\,e^\la$. If $\widehat G$ denotes the
Fourier transform of $G$, then it holds
\begin{align*}
F(\sqrt[m]L)
&=
G(R^{-m}L)\,e^{-R^{-m}L}
=
\frac1{\sqrt{2\pi}}
\int_{-\infty}^\infty\widehat G(\tau)e^{-(1-i\tau)R^{-m}L}\,d\tau
\end{align*}
in the strong convergence sense in $L^2(X)$. Thus, Lemma
\ref{DOS_Lem4_1} and the Cauchy-Schwarz inequality yield for any
$y\in X$, $s\geq0$, and $\eps>0$ whenever $F(R\cdot)\in
H_2^{(s+1)/2+\eps}$
\begin{align}\label{DOS_Lem(4.4)}
&\bigl\|F(\sqrt[m]L)\,\cf_{B(y,1/R)}\bigr\|_{L^2(X)\to
L^2(X,(1+Rd(\cdot,y))^sd\mu)}
\nonumber
\\&\quad\leqC
\int_{-\infty}^\infty|\widehat
G(\tau)|\,\bigl\|e^{-(1-i\tau)R^{-m}L}\cf_{B(y,1/R)}
 \bigr\|_{L^2(X)\to L^2(X,(1+Rd(\cdot,y))^sd\mu)}\,d\tau
\nonumber
\\&\quad\leqC
\int_{-\infty}^\infty|\widehat G(\tau)|\,(1+\tau^2)^{s/4}\,d\tau
\nonumber
\\&\quad\leq
\biggl(\int_{-\infty}^\infty
 |\widehat G(\tau)|^2\,(1+\tau^2)^{\frac{s+1+\eps}2}\,d\tau\biggr)^{1/2}
 \biggl(\int_{-\infty}^\infty(1+\tau^2)^{-\frac{1+\eps}2}\,d\tau\biggr)^{1/2}
\nonumber
\\&\quad\leqC
\|G\|_{H_2^{(s+1+\eps)/2}}\,.
\end{align}
Due to $\supp F(R\cdot)\subset[1/4,1]$, it follows
\begin{align}\label{DOS_Lem(4.5)}
\|G\|_{H_2^{(s+1+\eps)/2}}
\leqC
\|F(R\cdot)\|_{H_2^{(s+1+\eps)/2}}
\leqC
\|F(R\cdot)\|_{H_q^{(s+1+\eps)/2}}
\end{align}
for each $q\in[2,\infty]$. From (\ref{DOS_Lem(4.4)}) and
(\ref{DOS_Lem(4.5)}) we obtain part a) of the lemma.

Inserting (\ref{DOS_Lem(4.5)}) in (\ref{DOS_Lem(4.2)_o_Planch})
leads to a statement in which the required order of
differentiability of the function $F(R\cdot)$ is $1/2$ larger than
that of part b). In order to get rid of this additional $1/2$, we
make use of the interpolation procedure as described in \cite[p.\
455]{DOS} (see also \cite{MaMe}) based on the Plancherel condition
(\ref{DOS_Plancherel_Hardy}). By a simple scaling argument, we
first observe that the claimed bound (\ref{DOS_Lem(4.2)}) is
equivalent to the following estimate
\begin{align}\label{DOS_Lem(4.6)}
&\bigl\|H(R^{-1}\sqrt[m]L)\,\cf_{B(y,1/R)}\bigr\|_{L^2(X)\to
L^2(X,(1+Rd(\cdot,y))^sd\mu)}
\leqC
\|H\|_{H_q^{s/2+\eps}}
\end{align}
for any $\eps>0$, $s\geq2/q$, $R>0$, $y\in X$, and any bounded Borel
function $H\colon[0,\infty)\to\C$ with $\supp H\subset[1/4,1]$ and
$H\in H_q^{s/2+\eps}$.

\smallskip
For fixed $R>0$, $y\in X$, and $\ph\in L^2(X)$ with
$\supp\ph\subset B(y,1/R)$ and $\|\ph\|_{L^2}=1$ define
$$
K_{y,R,q}\colon E_q\to L^2(X)\,,
\quad H\mapsto H(R^{-1}\sqrt[m]L)\ph\,,
$$
where $E_q:=L^\infty([1/4,1])$ if $q<\infty$ and $E_q:=C^0([1/4,1])$
if $q=\infty$. According to the Plancherel condition
(\ref{DOS_Plancherel_Hardy}), we see after rescaling that
\begin{align*}
\|K_{y,R,q}(H)\|_{L^2}
\leqC
\|H\|_{L^q([1/4,1])}
\end{align*}
for every $H\in E_q$. Next, for $\al\geq0$ we denote by
$H_q^\al([1/4,1])$ the set of all $H\in H_q^\al$ with $\supp
H\subset[1/4,1]$. The inequalities (\ref{DOS_Lem(4.4)}) and
(\ref{DOS_Lem(4.5)}) lead to
$$
\|K_{y,R,q}(H)\|_{L^2(X,(1+Rd(\cdot,y))^sd\mu)}\leqC
\|H\|_{H_q^{(s+1+\eps)/2}([1/4,1])}
$$
for any $s\geq0$, $\eps>0$, and $H\in H_q^{(s+1+\eps)/2}([1/4,1])$.
Now complex interpolation yields for every $\theta\in(0,1)$
\begin{align}\label{erkhgaejbgke}
\|K_{y,R,q}(H)\|_{L^2(X,(1+Rd(\cdot,y))^{\theta s}d\mu)}
\leqC
\|H\|_{H_q^{(s+1+\eps)\theta/2+\delta}([1/4,1])}
\end{align}
for any $s\geq0$, $\eps>0$, $H\in
H_q^{(s+1+\eps)\theta/2+\del}([1/4,1])$, and $\delta>0$.

\smallskip
Let $s'\geq2/q$ and $\eps'>0$ be arbitrary. Take $\theta\in(0,1)$
and $\del>0$ with $(1+\eps)\theta/2+\delta=\eps'$. Next, choose
$s\geq0$ with $s\theta=s'$. Then inequality (\ref{erkhgaejbgke})
reads
$$
\|K_{y,R}(H)\|_{L^2(X,(1+Rd(\cdot,y))^{s'}d\mu)}
\leqC
\|H\|_{H_q^{s'/2+\eps'}([1/4,1])}
$$
for any $H\in H_q^{s'/2+\eps'}([1/4,1])$. Taking the supremum over
all $\ph\in L^2(X)$ such that $\supp\ph\subset B(y,1/R)$ and
$\|\ph\|_{L^2}=1$ yields
$$
\bigl\|H(R^{-1}\sqrt[m]L)\,\cf_{B(y,1/R)}\bigr\|_
{L^{2}(X)\to L^2(X,(1+Rd(\cdot,y))^{s'}d\mu)}
\leqC
\|H\|_{H_q^{s'/2+\eps'}([1/4,1])}
$$
for any $H\in H_q^{s'/2+\eps'}([1/4,1])$. This proves
(\ref{DOS_Lem(4.6)}) and thus (\ref{DOS_Lem(4.2)}).
\end{proof}

\begin{spezbew}{Theorem \ref{DY_Thm1.1}~a)}
Let $F\colon[0,\infty)\to\C$ be a bounded Borel function. Observe
that $F$ satisfies (\ref{HC_Bessel}) if and only if the function
$\la\mapsto F(\sqrt[m]\la)$ satisfies (\ref{HC_Bessel}). Hence, we
can consider $F(\sqrt[m]L)$ in lieu of $F(L)$ during the proof.
First, we write
$$
F(\sqrt[m]L)=(F-F(0))(\sqrt[m]L)+F(0)I
$$
and notice, after replacing $F$ by $F-F(0)$, that we may assume
$F(0)=0$ in the sequel. Due to the properties of $\om$, for every
$\la\geq0$ we then have the decomposition
$$
F(\la)
=
\sum_{l=-\infty}^\infty\om(2^{-l}\la)F(\la)
=
\sum_{l=-\infty}^\infty F_l(\la)\,,
$$
where $F_l(\la):=\om(2^{-l}\la)F(\la)$.

Fix $s>D/2$ and $M\in\N$ with $M>2s/m$. Further, assume that $F$
fulfills the Hörmander condition (\ref{HC_Bessel}) of order $s+1/2$.
For verifying the uniform boundedness of
$\sum_{l=-N}^NF_l(\sqrt[m]L)$ in $H_L^1(X)$, we apply Theorem
\ref{DY_Thm3.1}. To this end, we only need to check that condition
(\ref{DY(3.1)}) holds for the operator $\sum_{l=-N}^NF_l(\sqrt[m]L)$
with a constant $C_F$ independent of $N\in\N$.

For each $l\in\Z$ and $r>0$, we introduce the abbreviations
\begin{align*}
F_{r,M}(\la)&:=F(\la)(1-e^{-(r\la)^m})^M\,,\\
F_{r,M}^l(\la)&:=F_l(\la)(I-e^{-(r\la)^m})^M
=\om(2^{-l}\la)F(\la)(1-e^{-(r\la)^m})^M\,,
\end{align*}
where $\la\geq0$. In this notation, we may write
\begin{align}\label{DY(4.7)}
F(\sqrt[m]L)(I-e^{-r^mL})^M
=
F_{r,M}(\sqrt[m]L)
=
\lim_{N\to\infty}\sum_{l=-N}^NF^l_{r,M}(\sqrt[m]L)\,.
\end{align}
We choose $s'\in(D/2,s)$ and claim that for all
$j\in\N\setminus\{1\}$, $l\in\Z$, and balls $B\subset X$
of radius $r$
\begin{align}\label{DY(4.8)}
&\bigl\|\cf_{U_j(B)}F^l_{r,M}(\sqrt[m]L)\,\cf_B\bigr\|_{L^2\to L^2}
\leqC
C_{\om,s}2^{-js'}(2^lr)^{-s'}\min\bigl\{1,(2^lr)^{mM}\bigr\}
\max\bigl\{1,(2^lr)^{D/2}\bigr\}\,,
\end{align}
where $C_{\om,s}:=\sup_{n\in\Z}\|\om F(2^n\cdot)\|_{H_2^{s+1/2}}$
and the implicit constant depends only on $m,M,s$ and the constants
in the Davies-Gaffney and the doubling condition.

This, together with (\ref{DY(4.7)}), shows that for any
$j\in\N\setminus\{1\}$ and any ball $B\subset X$ of radius~$r$
\begin{align*}
&\bigl\|\cf_{U_j(B)}F(\sqrt[m]L)(I-e^{-r^mL})^M\cf_B\bigr\|_{L^2\to L^2}
\\&\quad\leqC
C_{\om,s}2^{-js'}\lim_{N\to\infty}\sum_{l=-N}^N(2^lr)^{-s'}
\min\bigl\{1,(2^lr)^{mM}\bigr\}\max\bigl\{1,(2^lr)^{D/2}\bigr\}
\\&\quad\leq
C_{\om,s}2^{-js'}\Biggl(\sum_{l\in\Z\colon2^lr>1}(2^lr)^{D/2-s'}+
\sum_{l\in\Z\colon2^lr\leq1}(2^lr)^{mM-s'}\Biggr)
\,.
\end{align*}

As both sums converge and have an upper bound independent of $r$,
the estimate (\ref{DY(3.1)}) holds for the function
$F(\sqrt[m]{\cdot\,})$, as desired.

\smallskip
It remains to prove our claim (\ref{DY(4.8)}). Consider a ball
$B\subset X$ with center $y\in X$ and radius $r>0$. First, we
observe that $\supp F_{r,M}^l\subset(2^{l-2},2^l)$.
Lemma \ref{DOS_Lem4_3aDG} a) then %with $R=2^l$
says that for any $l\in\Z$ and any $\eps>0$
\begin{align}\label{rwflkareflgfklrefg}
\bigl\|&F_{r,M}^l(\sqrt[m]L)\,\cf_{B(y,2^{-l})}
\bigr\|_{L^2(X)\to L^2(X,(1+2^ld(\cdot,y))^{2s'}d\mu)}
\leqC
\|F_{r,M}^l(2^l\cdot)\|_{H_2^{s'+1/2+\eps}}\,.
\end{align}
Let $j\in\N\setminus\{1\}$. For each $x\in U_j(B)$ we obtain, due to
$d(x,y)\geq2^{j-1}r$, the estimate
$(1+2^ld(x,y))^{s'}\geq2^{(j-1)s'}(2^lr)^{s'}$. Hence, we get for
$\eps:=s-s'>0$
\begin{align*}
&2^{-s'}2^{js'}(2^lr)^{s'}\bigl\|\cf_{U_j(B)}
F_{r,M}^l(\sqrt[m]L)\,\cf_{B(y,2^{-l})}\bigr\|_{L^2\to L^2}
\\&\qquad\leq
\bigl\|\cf_{U_j(B)}F_{r,M}^l(\sqrt[m]L)\,\cf_{B(y,2^{-l})}
\bigr\|_{L^2(X)\to L^2(X,(1+2^ld(\cdot,y))^{2s'}d\mu)}
\\&\qquad\leqC
\|F_{r,M}^l(2^l\cdot)\|_{H_2^{s+1/2}}
\end{align*}
or equivalently
\begin{align}\label{3riz3hottnz}
\bigl\|\cf_{U_j(B)}F_{r,M}^l(\sqrt[m]L)\,\cf_{B(y,2^{-l})}\bigl\|_{L^2\to L^2}
\leqC
2^{-js'}(2^lr)^{-s'}\|F_{r,M}^l(2^l\cdot)\|_{H_2^{s+1/2}}\,.
\end{align}
For $l\in\Z$ with $r\leq2^{-l}$ the left-hand side is an upper bound
for $\|\cf_{U_j(B)}F_{r,M}^l(\sqrt[m]L)\,\cf_{B}\|_{2\to2}$.

In the case $l\in\Z$ with $r>2^{-l}$, we cover $B=B(y,r)$ by balls
of radius $2^{-l}$. This procedure eventually leads to an additional
factor depending on the ratio of $r$ and $2^{-l}$ and the dimension
of the underlying space $X$. By Lemma \ref{ueberdeckung}, one can
construct a family of points $y_1,\ldots,y_K\in B(y,r)$ such that
$B(y,r)\subset\bigcup_{\nu=1}^KB(y_\nu,2^{-l})$, $K\leqC(2^lr)^D$,
and every $x\in B(y,r)$ is contained in at most $M$ balls
$B(y_\nu,2^{-l})$, where $M$ depends only on the constants in the
doubling condition. Observe that
$$
U_j(B(y,r))\subset\bigcup_{\eta=j-1}^{j+1}U_\eta(B(y_\nu,r))
$$
for all $j\in\N\setminus\{1\}$ and $\nu\in\{1,2,\ldots,K\}$. Therefore,
by (\ref{3riz3hottnz}), one obtains
\begin{align*}
\bigl\|\cf_{U_j(B(y,r))}F_{r,M}^l(\sqrt[m]L)\,\cf_{B(y_\nu,2^{-l})}\bigr\|_{
L^2\to L^2}
&\leq
\sum_{\eta=j-1}^{j+1}
\bigl\|\cf_{U_\eta(B(y_\nu,r))}F_{r,M}^l(\sqrt[m]L)\,\cf_{B(y_\nu,2^{-l})}
\bigr\|_{L^2\to L^2}
\\&\leqC
\sum_{\eta=j-1}^{j+1}2^{-\eta
s'}(2^lr)^{-s'}\|F_{r,M}^l(2^l\cdot)\|_{H_2^{s+1/2}}
\\&\leqC
2^{-js'}(2^lr)^{-s'}\|F_{r,M}^l(2^l\cdot)\|_{H_2^{s+1/2}}\,.
\end{align*}
Consider $g,h\in L^2(X)$ with $\|g\|_{L^2}=1$ and $\|h\|_{L^2}=1$. Then
we obtain for every $j\in\N\setminus\{1\}$ and every $l\in\Z$ with $r>2^{-l}$
\begin{align*}
\bigl|\bigl\langle
h,\cf_{U_j(B(y,r))}F_{r,M}^l(\sqrt[m]L)\,\cf_{B(y,r)}g\bigr\rangle\bigr|^2
&=
\bigl|\bigl\langle
\cf_{B(y,r)}F_{r,M}^l(\sqrt[m]L)^*\,\cf_{U_j(B(y,r))}h,g\bigr\rangle\bigr|^2
\\&\leq
\bigl\|\cf_{B(y,r)}F_{r,M}^l(\sqrt[m]L)^*\,\cf_{U_j(B(y,r))}h\bigr\|^2_{L^2}\,
\|g\|_{L^2}^2
\\&=
\int_{B(y,r)}\bigl|F_{r,M}^l(\sqrt[m]L)^*(\cf_{U_j(B(y,r))}h)(x)\bigr|^2\,
d\mu(x)
\\&\leq
\sum_{\nu=1}^K\int_{B(y_\nu,2^{-l})}\bigl|
F_{r,M}^l(\sqrt[m]L)^*(\cf_{U_j(B(y,r))}h)(x)\bigr|^2\,d\mu(x)
\\&\leq
\sum_{\nu=1}^K\bigl\|\cf_{U_j(B(y,r))}
F_{r,M}^l(\sqrt[m]L)\,\cf_{B(y_\nu,2^{-l})}\bigr\|_{L^2\to L^2}^2
\\&\leqC
\sum_{\nu=1}^K\bigl(2^{-js'}(2^lr)^{-s'}\|
F_{r,M}^l(2^l\cdot)\|_{H_2^{s+1/2}}\bigr)^2\,.
\end{align*}
Thus, by taking the supremum over all such $g,h$ and by recalling
$\sqrt K\leqC(2^lr)^{D/2}$, we deduce
$$
\bigl\|\cf_{U_j(B(y,r))}F_{r,M}^l(\sqrt[m]L)\,\cf_{B(y,r)}\bigr\|_{L^2\to L^2}
\leqC
(2^lr)^{D/2}\,2^{-js'}(2^lr)^{-s'}\|F_{r,M}^l(2^l\cdot)\|_{H_2^{s+1/2}}\,.
$$
In summary, we have shown that
\begin{align}\label{rwfklwbfkawf}
\bigl\|\cf_{U_j(B)}F_{r,M}^l(\sqrt[m]L)\,\cf_{B}\bigr\|_{L^2\to L^2}
\leqC
\max\bigl\{1,(2^lr)^{D/2}\bigr\}\,2^{-js'}(2^lr)^{-s'}\|
F_{r,M}^l(2^l\cdot)\|_{H_2^{s+1/2}}
\end{align}
for any $j\in\N\setminus\{1\}$, $l\in\Z$, and any ball
$B\subset X$ of radius $r$.

\smallskip
If $\ga$ is an integer larger than $s+1/2$, then it holds
\begin{align}\label{rewkjkbewrf}
\|F_{r,M}^l(2^l\cdot)\|_{H_2^{s+1/2}}\nonumber
&=
\bigl\|\la\mapsto\om(\la)F(2^l\la)(1-e^{-(2^lr\la)^m})^M\bigr\|_{H_2^{s+1/2}}
\\&\leqC\nonumber
\|\om F(2^l\cdot)\|_{H_2^{s+1/2}}
\bigl\|\la\mapsto(1-e^{-(2^lr\la)^m})^M\bigr\|_{C^\ga([\frac14,1])}
\\&\leqC
\sup_{n\in\Z}\|\om F(2^n\cdot)\|_{H_2^{s+1/2}}\min\bigl\{1,(2^lr)^{mM}\bigr\}\,.
\end{align}
The first inequality is due to \cite[Corollary (ii), p.\ 143]{Tri},
whereas the second inequality follows from \cite[Lemma 3.5]{B}.

In view of (\ref{rwfklwbfkawf}) and (\ref{rewkjkbewrf}), the claim
(\ref{DY(4.8)}) is confirmed. This completes the proof.
\end{spezbew}

The proof of Theorem~\ref{Hardy_Plancherel} follows the same lines
as that of Theorem~\ref{DY_Thm1.1}~a) with one small modification.
Instead of Lemma \ref{DOS_Lem4_3aDG}~a) one has to employ part b) of
the same lemma to obtain the desired regularity order in the
Hörmander condition.

\section{Boundedness of spectral multipliers on
\texorpdfstring{$\smash{H^p_L(X)}$}{H\^{}p(X)} and
\texorpdfstring{$\smash{L^p(X)}$}{L\^{}p(X)}}
\label{Chapter6Interpol}

In the preceding section we established spectral multiplier theorems on
the Hardy space $H^1_L(X)$ which ensure the boundedness of the
operator $F(L)$ on $H^1_L(X)$, where $F$ is a bounded Borel function
satisfying \eqref{HC_Bessel} or \eqref{HC_Hoelder} and $L$ is an
injective, non-negative, self-adjoint operator on $L^2(X)$ for which
Davies-Gaffney estimates hold. Since self-adjoint operators on
$L^2(X)$ have the functional calculus for arbitrary bounded Borel
functions $\R\to\C$ without any regularity hypothesis, one expects
that the regularity assumptions on $F$ can be weakened when one asks
about boundedness of $F(L)$ on $H^p_L(X)$ for some $p\in(1,2)$. This
is actually true, as the interpolation procedure described in
\cite[Section 4.6.1]{CK} shows.

\begin{definition}
Let $p\in[1,2]$, $q\in[2,\infty]$, $s>1/q$, and $L$ be an injective,
non-negative, self-adjoint operator on $L^2(X)$ which fulfills
Davies-Gaffney estimates. We say that $L$ has an {\em $\mathcal
H_q^s$ Hörmander calculus on $H_L^p(X)$} if there exists a constant
$C>0$ such that
$$
\|F(L)\|_{H_L^p(X)\to H_L^p(X)}\leq C\sup_{n\in\Z}\|\om
F(2^n\cdot)\|_{H_q^s}
$$
for all $F\in\mathcal H_q^s:=\{F\colon(0,\infty)\to\C \mbox{ bounded
Borel function such that }
\sup_{n\in\Z}\|\om F(2^n\cdot)\|_{H^s_q}<\infty\}$.

Since the Hörmander condition $\sup_{n\in\Z}\|\om
F(2^n\cdot)\|_{H^s_q}<\infty$ contains no information on the value
of $F(0)$, the value $F(0)$ is not regarded in the so-called {\em
Hörmander class} $\mathcal H_q^s$. But this causes no problems as
long as one studies injective operators.
\end{definition}

The interpolation statement concerning the Hörmander calculus,
adapted to our present situation, reads as follows (cf.\
\cite[Corollary~4.84]{CK}).

\begin{fact}\label{Interpolation-von-Christoph}
Let $L$ be an injective, non-negative, self-adjoint operator on
$L^2(X)$ such that Davies-Gaffney estimates DG$_m$ hold for some
$m\geq2$. Assume that $L$ has an $\mathcal H_q^s$ Hörmander calculus
on the Hardy space $H^1_L(X)$ for some $q\geq2$ and $s>1/q$. Then,
for any $\theta\in(0,1)$, the operator $L$ has an $\mathcal
H_{q_\theta}^{s_\theta}$ Hörmander calculus on
$[L^2(X),H_L^1(X)]_\theta$ whenever $s_\theta>\theta s$ and
$q_\theta>q/\theta$.
\end{fact}

With the help of this interpolation result, we obtain spectral
multiplier theorems on the Hardy space $H_L^p(X)$ for each
$p\in[1,2]$. We also state a version including the Plancherel
condition which yields a lower regularity order in the Hörmander
condition.

\begin{theorem}\label{mainresHardy}
Let $L$ be an injective, non-negative, self-adjoint operator on
$L^2(X)$ satisfying Davies-Gaffney estimates DG$_{m}$ for some
$m\geq2$.
\begin{enumerate}
\item[\bf a)]
Fix $p\in[1,2]$. Let $s>(D+1)(1/p-1/2)$ and $1/q<1/p-1/2$. Then $L$
has an $\mathcal H_q^s$ Hörmander calculus on $H_L^p(X)$, i.e.\ for
every bounded Borel function $F\colon(0,\infty)\to\C$ with
$\sup_{n\in\Z}\|\om F(2^n\cdot)\|_{H_q^s}<\infty$, there exists a
constant $C>0$ such that
$$
\|F(L)\|_{H_L^p(X)\to H_L^p(X)}\leq C\sup_{n\in\Z}\|\om
F(2^n\cdot)\|_{H_q^s}\,.
$$

\item[\bf b)]
Let $p\in[1,2]$ and $s>D(1/p-1/2)$. Then $L$ has an $\mathcal
H_\infty^s$ Hörmander calculus on $H_L^p(X)$, i.e.\ for every
bounded Borel function $F\colon(0,\infty)\to\C$ with
$\sup_{n\in\Z}\|\om F(2^n\cdot)\|_{C^s}<\infty$, there exists a
constant $C>0$ such that
$$
\|F(L)\|_{H_L^p(X)\to H_L^p(X)}\leq C\sup_{n\in\Z}\|\om
F(2^n\cdot)\|_{C^s}\,.
$$

\item[\bf c)]
Assume further that $L$ fulfills the Plancherel condition
(\ref{DOS_Plancherel_Hardy}) for some $q_0\in[2,\infty)$. Fix
$p\in[1,2]$. Let $s>\max\{D,2/q_0\}\,(1/p-1/2)$ and
$1/q<2/q_0\,(1/p-1/2)$. Then $L$ has an $\mathcal H_q^s$ Hörmander
calculus on $H_L^p(X)$, i.e.\ for every bounded Borel function
$F\colon(0,\infty)\to\C$ with $\sup_{n\in\Z}\|\om
F(2^n\cdot)\|_{H_q^s}<\infty$, there exists a constant $C>0$ such
that
$$
\|F(L)\|_{H_L^p(X)\to H_L^p(X)}\leq C\sup_{n\in\Z}\|\om
F(2^n\cdot)\|_{H_q^s}\,.
$$
\end{enumerate}
\end{theorem}

\begin{proof}
Let $p\in[1,2]$. The assertion of part a) follows directly by
combining Theorem~\ref{DY_Thm1.1}~a) and Fact
\ref{Interpolation-von-Christoph} with $\theta:=2(1/p-1/2)$. A
similar reasoning using Theorem~\ref{DY_Thm1.1}~b) and the embedding
$C^{s+\eps}\hookrightarrow[C^{s_0},C^{s_1}]_\theta$ (for $\eps>0$,
$s_0<s_1$, $\theta\in(0,1)$ and $s=(1-\theta)s_0+\theta s_1$) gives
part b).

\smallskip
Suppose that $L$ additionally satisfies the Plancherel condition
(\ref{DOS_Plancherel_Hardy}) for some $q_0\in[2,\infty)$. Due to 
Theorem~\ref{Hardy_Plancherel}, $L$ has an $\mathcal H_{q_0}^s$
Hörmander calculus on $H^1_L(X)$ for each $s>\max\{D/2,1/q_0\}$. Now
Fact \ref{Interpolation-von-Christoph} with $\theta:=2(1/p-1/2)$
yields the assertion of c).
\end{proof}

If the operator $L$ actually enjoys generalized Gaussian estimates
GGE$_m(p_0,p_0')$ for some $p_0\in[1,2)$ and $m\geq2$, then 
Theorem~\ref{HpLp} ensures $H^p_L(X)=L^p(X)$ for every $p\in(p_0,2]$.
Therefore, we deduce from Theorem~\ref{mainresHardy} spectral 
multiplier results on $L^p(X)$ as well. The regularity assumptions 
in our statement a) are weaker than those of \cite[Theorem~1.1]{B} 
and \cite[Theorem~5.6]{KriPre} (or \cite[Theorem~4.95]{CK}), 
where $s>(D+1)/2$, $q=2$ and $s>D|1/p-1/2|+1/2$, $q=2$ were required,
respectively.

\begin{theorem}\label{mainres}
Let $L$ be a non-negative, self-adjoint operator on $L^2(X)$ such
that generalized Gaussian estimates GGE$_m(p_0,p_0')$ hold for some
$p_0\in[1,2)$ and $m\geq2$.
\begin{enumerate}
\item[\bf a)]
For fixed $p\in(p_0,p_0')$ suppose that $s>(D+1)|1/p-1/2|$ and
$1/q<|1/p-1/2|$. Then, for every bounded Borel function
$F\colon[0,\infty)\to\C$ with $\sup_{n\in\Z}\|\om
F(2^n\cdot)\|_{H_q^s}<\infty$, the operator $F(L)$ is bounded on
$L^p(X)$. More precisely, there exists a constant $C>0$ such that
$$
\|F(L)\|_{L^p\to L^p}\leq C\Bigl(\sup_{n\in\Z}\|\om
F(2^n\cdot)\|_{H_q^s}+|F(0)|\Bigr)\,.
$$

\item[\bf b)]
Let $p\in(p_0,p_0')$ and $s>D|1/p-1/2|$. Then, for any bounded
Borel function $F\colon[0,\infty)\to\C$ with $\sup_{n\in\Z}\|\om
F(2^n\cdot)\|_{C^s}<\infty$, the operator $F(L)$ is bounded on
$L^p(X)$. More precisely, there exists a constant $C>0$ such that
$$
\|F(L)\|_{L^p\to L^p}\leq C\Bigl(\sup_{n\in\Z}\|\om
F(2^n\cdot)\|_{C^s}+|F(0)|\Bigr)\,.
$$

\item[\bf c)]
In addition, assume that $L$ fulfills the Plancherel condition
(\ref{DOS_Plancherel_Hardy}) for some $q_0\in[2,\infty)$. Fix
$p\in(p_0,p_0')$. Let $s>\max\{D,2/q_0\}\,|1/p-1/2|$ and
$1/q<2/q_0\,|1/p-1/2|$. Then, for every bounded Borel function
$F\colon[0,\infty)\to\C$ with $\sup_{n\in\Z}\|\om
F(2^n\cdot)\|_{H_q^s}<\infty$, the operator $F(L)$ is bounded on
$L^p(X)$. More precisely, there exists a constant $C>0$ such that
$$
\|F(L)\|_{L^p\to L^p}\leq C\Bigl(\sup_{n\in\Z}\|\om
F(2^n\cdot)\|_{H_q^s}+|F(0)|\Bigr)\,.
$$
\end{enumerate}
\end{theorem}

\begin{proof}
Let $p\in(p_0,2)$. We shall prove the three assertions
simultaneously. Suppose that $s>(D+1)(1/p-1/2)$ and $1/q<1/p-1/2$
for the proof of a) and $s>D(1/p-1/2)$, $q=\infty$ for b). For the 
proof of part c) suppose that $L$ fulfills the Plancherel condition
(\ref{DOS_Plancherel_Hardy}) for some $q_0\in[2,\infty)$ as well as
$s>\max\{D,2/q_0\}\,(1/p-1/2)$ and $1/q<2/q_0\,(1/p-1/2)$.

Since injectivity of $L$ is not assumed, Theorem \ref{mainresHardy}
cannot be applied directly. In order to overcome this difficulty, we
use the concept of \cite[Proposition 15.2]{Levico} (see also
\cite[Theorem 3.8]{CDMY}) that provides a decomposition of the space
$L^2(X)$ as the orthogonal sum of the closure of the range
$\ov{R(L)}$ of $L$ and the null space $N(L)$ of $L$. The operator
$L$ then takes the form
$$
L=\begin{pmatrix}L_0&0\\0&0\end{pmatrix}
$$
with respect to the decomposition $L^2(X)=\ov{R(L)}\oplus N(L)$,
where $L_0$ is the part of $L$ in $\ov{R(L)}$, i.e.\ the restriction
of $L$ to $\dom(L_0):=\{x\in\ov{R(L)}\cap\dom(L):\,
Lx\in\ov{R(L)}\}$. But $L_0$ is injective on its domain, so that
Theorem \ref{mainresHardy} applies to $L_0$. This approach was
already made in \cite[Section 4.6.1]{CK} and, as remarked in
\cite[Illustration 4.87]{CK}, the decomposition and the
interpolation result cited in Fact \ref{Interpolation-von-Christoph}
can be combined. Hence, $L_0$ has an $\mathcal H_q^s$ Hörmander
calculus on $H_{L_0}^p(X)$. Consider a bounded Borel function
$F\colon[0,\infty)\to\C$ with $\sup_{n\in\Z}\|\om
F(2^n\cdot)\|_{H_q^s}<\infty$. Then it holds
$$
F(L)=\begin{pmatrix}(F|_{(0,\infty)})(L_0)&0\\0&F(0)\,
I_{N(L)}\end{pmatrix}
$$
on $H_{L_0}^p(X)\cap L^2(X)$. Because of $F|_{(0,\infty)}\in\mathcal
H_q^s$, one has moreover
\begin{align*}
\|(F|_{(0,\infty)})(L_0)\|_{H_{L_0}^p(X)\to
H_{L_0}^p(X)}&\leqC\sup_{n\in\Z}\|\om F(2^n\cdot)\|_{H_q^s}
\intertext{as well as} \|F(0)\,I_{N(L)}\|_{H_{L_0}^p(X)\to
H_{L_0}^p(X)}&\leq|F(0)|\,.
\end{align*}
Since, by Theorem~\ref{HpLp}, the spaces $H^p_{L_0}(X)$ and $L^p(X)$
coincide, the statements a), b) and c) are proven for any
$p\in(p_0,2)$.

Let $p\in(2,p_0')$. Due to the self-adjointness of $L$ on $L^2(X)$,
boundedness of spectral multipliers on $L^p(X)$ follows by the case
proved above and dualization. The claim for $p=2$ is trivial.
\end{proof}

\begin{remark}
(1) The assertions of Theorem \ref{mainres} remain even valid for 
open subsets $\Om$ of $X$ provided that the ball appearing on the 
right-hand side of (\ref{GGE}) is the one in $X$. The reasoning is 
standard and relies on an observation quoted in 
\cite[pp.~934-935]{BK2} by adapting the arguments given in 
\cite[p.~245]{DM99} (see also \cite[p.~452]{B}). For this purpose,
one has only to extend an operator $T\colon L^p(\Om)\to L^q(\Om)$ 
by zero to the operator $\wt T\colon L^p(X)\to L^q(X)$ defined via
$$
\wt Tu(x):=\left\{
\begin{array}{cl}
T(\cf_\Om u)(x)&~\mbox{for }x\in\Om\\
0&~\mbox{for }x\in X\setminus\Om
\end{array}
\right.\qquad(u\in L^p(X),\,\mu\mbox{-a.e.}~x\in X)
$$
and observe that $\|\wt T\|_{L^p(X)\to L^q(X)}=\|T\|_{L^p(\Om)\to
L^q(\Om)}$. The modified result allows to cover elliptic operators
on irregular domains $\Om\subset\R^D$ as well (cf.\ e.g.\
\cite[Section 2.1]{B}).

(2) Of course, it is possible to apply the same method (complex
interpolation with the functional calculus in $L^2(X)$ and coincidence
of $H_L^p(X)$ and $L^p(X)$) also with Theorem~\ref{DY_Thm1.1_Plancherel}
as a starting point. We do not go into details here.
\end{remark}

\section{Proofs of some auxiliary results}
\label{proofs}

In this section, we proof the Lemmata \ref{GGEkomplexbelRadius},
\ref{DLProp2.3}, \ref{LemDY(3.1)VarianteLem} and \ref{intVerReg}.

\begin{spezbew}{Lemma \ref{GGEkomplexbelRadius}}
{\bf a)}\hskip\labelsep In view of Fact \ref{GGEkomplex}, there are
constants $b,C>0$ such that
\begin{align*}
&\bigl\|\cf_{B(x,r_z)}e^{-zL}\cf_{B(y,r_z)}\bigr\|_{L^p\to L^q}
\leq
C\,|B(x,r_z)|^{-(\frac1{p}-\frac1{q})}
\biggl(\frac{|z|}{\Re z}\biggr)^{D(\frac1{p}-\frac1{q})}
\exp\Biggl(-b\biggl(\frac{d(x,y)}{r_z}\biggr)^{\frac{m}{m-1}}\Biggr)
\end{align*}
for all $x,y\in X$ and $z\in\C$ with $\Re z>0$. By Fact
\ref{GGEequiv} (with $T:=({|z|}/{\Re z})^{-D(1/{p}-1/{q})}e^{-zL}$),
one finds $b',C'>0$ such that
\begin{align}\label{ergakfsrfs}
\bigl\|\cf_{B_1}v_{r_z}^{\frac1{p}-\frac1{q}}T\cf_{B_2}\bigr\|_{L^p\to L^q}
\leq
C'\exp\Biggl(-b'\biggl(\frac{\dist(B_1,B_2)}{r_z}\biggr)^{\frac{m}{m-1}}\Biggr)
\end{align}
for all balls $B_1,B_2\subset X$ and all $z\in\C$ with $\Re z>0$, 
where $v_{r_z}:=|B(\cdot,r_z)|$.
Let $r>0$ be fixed. The doubling property leads to
\begin{align*}
v_{r}(x)
\leqC
\biggl(1+\frac{r}{r_z}\biggr)^Dv_{r_z}(x)
\end{align*}
for every $x\in X$ and $z\in\C$ with $\Re z>0$. Now choose arbitrary
$x,y\in X$ with $d(x,y)\geq4r$ and consider the balls $B_1:=B(x,r)$
and $B_2:=B(y,r)$. Then it holds
$$
\dist(B_1,B_2)
=
d(x,y)-2r
\geq
\frac12\,d(x,y)\,.
$$
By inserting $B_1$, $B_2$ into (\ref{ergakfsrfs}) and collecting the
estimates above together, one arrives at
\begin{align*}
&\bigl\|\cf_{B(x,r)}v_{r}^{\frac1{p}-\frac1q}e^{-zL}\cf_{B(y,r)}\bigr\|_{L^p\to L^q}
\leqC
\biggl(1+\frac{r}{r_z}\biggr)^{D(\frac1{p}-\frac1{q})}
\biggl(\frac{|z|}{\Re z}\biggr)^{D(\frac1{p}-\frac1{q})}
\exp\Biggl(-b'\biggl(\frac{d(x,y)}{2r_z}\biggr)^{\frac m{m-1}}\Biggr)\,.
\end{align*}
Since $v_{r}(x)\cong v_{r}(z)$ for all $z\in B(x,r)$ (cf.\ Fact
\ref{BKKugel}), one obtains the desired estimate
\begin{align*}
&\bigl\|\cf_{B(x,r)}e^{-zL}\cf_{B(y,r)}\bigr\|_{L^p\to L^q}
\\&\quad\leq
C'\,|B(x,r)|^{-(\frac1{p}-\frac1{q})}
\biggl(1+\frac{r}{r_z}\biggr)^{D(\frac1{p}-\frac1{q})}
\biggl(\frac{|z|}{\Re z}\biggr)^{D(\frac1{p}-\frac1{q})}
\exp\Biggl(-b'\biggl(\frac{d(x,y)}{2r_z}\biggr)^{\frac m{m-1}}\Biggr)
\end{align*}
for all $r>0$, $z\in\C$ with $\Re z>0$, and $x,y\in X$ with
$d(x,y)\geq4r$. By the cost of changing the constants $b',C'$, one
is able to remove this restriction on $d(x,y)$.

\smallskip
\noindent{\bf b)}\hskip\labelsep Our approach mimics that of
\cite[(i)$\Rightarrow$(3), p.\ 359]{BKLeg}. Observe that it suffices
to prove the statement only for every $k\in\N\setminus\{1\}$. With
the help of \cite[Lemma 3.4]{BKLeg}, we can write for each
$k\in\N\setminus\{1\}$, $r>0$, $x\in X$, and $z\in\C$ with $\Re z>0$
\begin{align*}
&\bigl\|\cf_{B(x,r)}e^{-zL}\cf_{A(x,r,k)}\bigr\|_{L^p\to L^q}
\\&\quad\leqC
\int_X
\bigl\|\cf_{B(x,r)}e^{-zL}\cf_{B(y,r)}\bigr\|_{L^p\to L^q}
\bigl\|\cf_{B(y,r)}\cf_{A(x,r,k)}\bigr\|_{L^q\to L^q}v_r(y)^{-1}\,d\mu(y)
\\&\quad=
\int_{B(x,(k+2)r)\setminus B(x,(k-1)r)}
\bigl\|\cf_{B(x,r)}e^{-zL}\cf_{B(y,r)}\bigr\|_{L^p\to L^q}v_r(y)^{-1}\,d\mu(y)\,.
\intertext{By exploiting the bound from part a), we continue our estimation}
&\quad\leqC
|B(x,r)|^{-(\frac1{p}-\frac1{q})}\biggl(1+\frac{r}{r_z}\biggr)^{D(\frac1{p}
-\frac1{q})}
\biggl(\frac{|z|}{\Re z}\biggr)^{D(\frac1{p}-\frac1{q})}
~\times\\&\qquad\quad\times~
\int_{B(x,(k+2)r)\setminus B(x,(k-1)r)}
\exp\Biggl(-b'\biggl(\frac{d(x,y)}{r_z}\biggr)^{\frac{m}{m-1}}\Biggr)v_r(y)^{-1}
\,d\mu(y)\,.
\intertext{Using $d(x,y)\geq(k-1)r\geq kr/2$ as well as
$v_r(y)^{-1}\leqC(k+2)^Dv_{(k+2)r}(y)^{-1}$ leads to}
&\quad\leqC
|B(x,r)|^{-(\frac1{p}-\frac1{q})}\biggl(1+\frac{r}{r_z}\biggr)^{D(\frac1{p}
-\frac1{q})}
\biggl(\frac{|z|}{\Re z}\biggr)^{D(\frac1{p}-\frac1{q})}
~\times\\&\qquad\quad\times~
\int_{B(x,(k+2)r)\setminus B(x,(k-1)r)}
\exp\Biggl(-2^{-\frac{m}{m-1}}b'\biggl(\frac{kr}{r_z}\biggr)^{\frac{m}{m-1}}
\Biggr)(k+2)^Dv_{(k+2)r}(y)^{-1}\,d\mu(y)
\\&\quad\leqC
|B(x,r)|^{-(\frac1{p}-\frac1{q})}\biggl(1+\frac{r}{r_z}\biggr)^{D(\frac1{p}
-\frac1{q})}
\biggl(\frac{|z|}{\Re z}\biggr)^{D(\frac1{p}-\frac1{q})}
~\times\\&\qquad\quad\times~
(k+2)^D\exp\Biggl(-2^{-\frac{m}{m-1}}b'\biggl(\frac{kr}{r_z}\biggr)^{\frac{m}{
m-1}}\Biggr)\,,
\end{align*}
where the last inequality is thanks to (\ref{doubl}). This proves
the statement.
\end{spezbew}

\begin{spezbew}{Lemma \ref{DLProp2.3}}
Let $K\in\N$ and $t>0$ be arbitrary. The Cauchy formula gives the
representation
$$
(tL)^{K}e^{-tL}
=
t^{K}\,\frac{(-1)^{K}K!}{2\pi i}
\int_{|z-t|=\eta t}e^{-zL}\,\frac{dz}{(z-t)^{K+1}}\,,
$$
where $\eta:=1/2\sin(\theta/2)$ for some $\theta\in(0,\pi/2)$. Note
that the choice of $\eta$ ensures that the ball
$\{z\in\C:|z-t|\leq\eta t\}$ is contained in the sector
$\Si_\theta:=\{z\in\C\setminus\{0\}:|\arg z|<\theta\}$. According to
Lemma \ref{GGEkomplexbelRadius}, it holds for every $x,y\in X$:
\begin{align*}
&\bigl\|\cf_{B(x,t^{1/m})}(tL)^{K}e^{-tL}\cf_{B(y,t^{1/m})}\bigr\|_{L^2\to L^2}
\\&\quad\leq
t^{K}\,\frac{K!}{2\pi}
\int_{|z-t|=\eta t}
\bigl\|\cf_{B(x,t^{1/m})}e^{-zL}\cf_{B(y,t^{1/m})}\bigr\|_{L^2\to L^2}
\,\frac{|dz|}{|z-t|^{K+1}}
\\&\quad\leqC
t^{K}\,\frac{K!}{2\pi} \int_{|z-t|=\eta t}
\exp\Biggl(-b'\biggl(\frac{d(x,y)}{r_z}\biggr)^{\frac m{m-1}}\Biggr)
\,\frac{|dz|}{(\eta t)^{K+1}}\,,
\intertext {where $r_z:=(\Re z)^{1/m}|z|/\Re z$. Due to $\Re
z\in[(1-\eta)t,(1+\eta)t]$ and $1\leq |z|/\Re z\leq1/\cos\theta$
for all $z$ belonging to the integration path, we have $r_z\cong
t^{1/m}$ with implicit constants depending only on $\theta$ or $m$.
Thus, we can finish our estimation as follows} &\quad\leqC
t^{K}\,\frac{K!}{2\pi}\,2\pi\eta t\,
\exp\Biggl(-b'\biggl(\frac{d(x,y)}{t^{1/m}}\biggr)^{\frac
m{m-1}}\Biggr) \frac{1}{(\eta t)^{K+1}}
\\&\quad=
\frac{K!}{\eta^K}\,
\exp\Biggl(-b'\biggl(\frac{d(x,y)}{t^{1/m}}\biggr)^{\frac m{m-1}}\Biggr)\,.
\end{align*}
\end{spezbew}

\begin{spezbew}{Lemma \ref{LemDY(3.1)VarianteLem}}
It suffices to check (\ref{DY(3.1)Variante}) only for each
$i,j\in\N\setminus\{1\}$ with $|j-i|>3$ since otherwise
(\ref{DY(3.1)Variante}) is valid by the spectral theorem after
choosing appropriate constants. Due to the self-adjointness of $L$,
one can swap $i$ and $j$ in the term on the left-hand side of
(\ref{DY(3.1)Variante}). Hence, it will be enough to show the
assertion for every $i,j\in\N\setminus\{1\}$ with $j-i>3$. By
applying \cite[Lemma 3.4]{BKLeg}, (\ref{DY(3.1)}), and the doubling
property, we get for each $r>0$ and each $x\in X$:
\begin{align*}
&\bigl\|\cf_{U_j(B(x,r))}F(L)(I-e^{-r^mL})^M\cf_{U_i(B(x,r))}\bigr\|_{L^2\to L^2}
\\&\quad\leqC
\int_X\bigl\|\cf_{U_j(B(x,r))}F(L)(I-e^{-r^mL})^M\cf_{B(z,r)}\bigr\|_{L^2\to L^2}
\bigl\|\cf_{B(z,r)}\cf_{U_i(B(x,r))}\bigr\|_{L^2\to L^2}\,\frac{d\mu(z)}{|B(z,r)|}
\\&\quad\leq
\int_{B(x,2^{i+1}r)\setminus B(x,2^{i-2}r)}
\sum_{\nu=j-i-3}^{j+i+1}\bigl\|\cf_{U_\nu(B(z,r))}F(L)(I-e^{-r^mL})^M\cf_{B(z,r)
}\bigr\|_{L^2\to L^2}
\,\frac{d\mu(z)}{|B(z,r)|}
\\&\quad\leqC
\int_{B(x,2^{i+1}r)}
\sum_{\nu=j-i-3}^{j+i+1}C_F\,2^{-\nu\del}\,2^{(i+1)D}\,
\frac{d\mu(z)}{|B(z,2^{i+1}r)|}\,.
\end{align*}
In the second step we covered $U_j(B(x,r))$ by dyadic annuli around
the point $z$. Here, we used, among other things, the elementary 
inequalities
\begin{align}\label{rgfrejfhlreb}
|2^{\al}-2^{\beta}|\geq2^{|\al-\beta|-1}
\qquad\mbox{and}\qquad
2^\al+2^\beta\leq2^{\al+\beta+1}
\end{align}
which are valid for each $\al,\beta\in\N_0$ with $\al\ne\beta$.
With the help of
\begin{align*}
\sum_{\nu=j-i-3}^{j+i+1}2^{-\nu\del}
=
2^{3\del}\,2^{-(j-i)\del}\,\sum_{\eta=0}^{2i+4}2^{-\eta\del}
\leqC
2^{-(j-i)\del}
\end{align*}
and Fact~\ref{BKKugel}, we finish our estimation as follows
\begin{align*}
&\bigl\|\cf_{U_j(B(x,r))}F(L)(I-e^{-r^mL})^M\cf_{U_i(B(x,r))}\bigr\|_{L^2\to L^2}
\\&\quad\leqC
C_F2\,^{-(j-i)\del}
\int_{B(x,2^{i+1}r)}2^{(i+1)D}\frac{d\mu(z)}{|B(z,2^{i+1}r)|}
\\&\quad\leqC
C_F\,2^{iD}\,2^{-(j-i)\del}\,.
\end{align*}
\end{spezbew}

\begin{spezbew}{Lemma \ref{intVerReg}}
Let $K,M\in\N$, $r>0$, and $x\in X$. At the beginning, we note that
the operator $P_{m,M,r}(L)$ is bounded on $L^2(X)$:
\begin{align*}
\bigl\|P_{m,M,r}(L)\bigr\|_{L^2\to L^2}
&\leq
r^{-m}\int_r^{\sqrt[m]2r}s^{m-1}\bigl\|I-e^{-s^mL}\bigr\|^M_{L^2\to L^2}\,ds
\\&\leq
r^{-m}\int_r^{\sqrt[m]2r}s^{m-1}2^M\,ds
=
\frac{2^M}m\,.
\end{align*}

With analogous arguments as in the proof of Lemma
\ref{LemDY(3.1)VarianteLem}, it is enough to verify
(\ref{intVerRegDG}) for each $i,j\in\N_0$ with $j-i>6$. To this
purpose, fix $k\in\{1,\ldots,M\}$ and $s\in[r,\sqrt[m]2r]$ for a
moment. We shall establish the estimate
\begin{align}\label{lrfrqrf}
\bigl\|\cf_{U_j(B(x,r))}e^{-ks^mL}\cf_{U_i(B(x,r))}\bigr\|_{L^2\to L^2}
\leq C
\exp\bigl(-b(2^{j-1}-2^{i+2})\bigr)
\end{align}
for some constants $b,C>0$ depending only on $m$, $M$ and the
constants in the Davies-Gaffney or doubling condition, but not on
the other parameters.

\smallskip
From the Davies-Gaffney estimates DG$_m$ we obtain for each $y\in X$:
\begin{align*}
&\bigl\|\cf_{B(x,r)}e^{-ks^mL}\cf_{B(y,r)}\bigr\|_{L^2\to L^2}
\leq
\bigl\|\cf_{B(x,k^{1/m}s)}e^{-ks^mL}\cf_{B(y,k^{1/m}s)}\bigr\|_{L^2\to L^2}
\\[1ex]&\qquad\leqC
\exp\Biggl(-b\biggl(\frac{d(x,y)}{k^{1/m}s}\biggr)^{\frac m{m-1}}\Biggr)
\leq
\exp\Biggl(-b(2M)^{-\frac1{m-1}}
\biggl(\frac{d(x,y)}{r}\biggr)^{\frac m{m-1}}\Biggr)\,.
\end{align*}
Therefore, Fact \ref{GGEequiv} yields for any $\nu\in\N$:
\begin{align*}
\bigl\|\cf_{A(x,r,\nu)}e^{-ks^mL}\cf_{B(x,r)}\bigr\|_{L^2\to L^2}
\leqC
\exp\bigl(-b\nu^{\frac m{m-1}}\bigr)
\leq
e^{-b\nu}\,.
\end{align*}
By applying \cite[Lemma 3.4]{BKLeg} and the doubling property, we
deduce
\begin{align*}
&\bigl\|\cf_{U_j(B(x,r))}e^{-ks^mL}\cf_{U_i(B(x,r))}\bigr\|_{L^2\to L^2}
\\&\quad\leqC
\int_X\bigl\|\cf_{U_j(B(x,r))}e^{-ks^mL}\cf_{B(z,r)}\bigr\|_{L^2\to L^2}
\bigl\|\cf_{B(z,r)}\cf_{U_i(B(x,r))}\bigr\|_{L^2\to L^2}\,\frac{d\mu(z)}{|B(z,r)|}
\\&\quad\leq
\int_{B(x,2^{i+1}r)\setminus B(x,2^{i-2}r)}
\sum_{\nu=2^{j-1}-2^{i+1}}^{2^j+2^{i+1}}
\bigl\|\cf_{A(z,r,\nu)}e^{-ks^mL}\cf_{B(z,r)}\bigr\|_{L^2\to L^2}\,
\frac{d\mu(z)}{|B(z,r)|}
\\&\quad\leqC
\int_{B(x,2^{i+1}r)}\sum_{\nu=2^{j-1}-2^{i+1}}^{2^j+2^{i+1}}e^{-b\nu}\,
2^{(i+1)D}\,\frac{d\mu(z)}{|B(z,2^{i+1}r)|}\,.
\end{align*}
With the help of
\begin{align*}
\sum_{\nu=2^{j-1}-2^{i+1}}^{2^j+2^{i+1}}e^{-b\nu}
&\leq
\exp\bigl(-b(2^{j-1}-2^{i+1})\bigr)\sum_{\eta=0}^{\infty}e^{-b\eta}
=
\frac{1}{1-e^{-b}}\,\exp\bigl(-b(2^{j-1}-2^{i+1})\bigr)
\end{align*}
and Fact \ref{BKKugel}, we finally arrive at the claimed estimate
(\ref{lrfrqrf})
\begin{align*}
\bigl\|\cf_{U_j(B(x,r))}e^{-ks^mL}\cf_{U_i(B(x,r))}\bigr\|_{L^2\to L^2}
\leqC
2^{iD}
\exp\bigl(-b(2^{j-1}-2^{i+1})\bigr)
\leqC
\exp\bigl(-b(2^{j-1}-2^{i+2})\bigr)\,.
\end{align*}

\smallskip
In view of the formula
$$
(I-e^{-s^mL})^M=\sum_{k=0}^M\binom{M}{k}(-1)^ke^{-ks^mL}
$$
and the disjointness of $U_i(B(x,r))$ and $U_j(B(x,r))$, we get from
(\ref{lrfrqrf})
\begin{align}\label{rgkldghdhdhgdhhf}
&\bigl\|\cf_{U_j(B(x,r))}P_{m,M,r}(L)\,\cf_{U_i(B(x,r))}\bigr\|_{L^2\to L^2}\nonumber
\\&\quad\leq\nonumber
\sum_{k=0}^M\binom{M}{k}
r^{-m}\int_r^{\sqrt[m]2r}s^{m-1}\bigl\|\cf_{U_j(B(x,r))}
 e^{-ks^mL}\cf_{U_i(B(x,r))}\bigr\|_{L^2\to L^2}\,ds
\\&\quad\leqC\nonumber
\sum_{k=1}^M\binom{M}{k}
r^{-m}\int_r^{\sqrt[m]2r}s^{m-1}\,ds\,
\exp\bigl(-b(2^{j-1}-2^{i+2})\bigr)
\\&\quad\leqC
\exp\bigl(-b(2^{j-1}-2^{i+2})\bigr)\,.
\end{align}
Due to the inequality (\ref{rgfrejfhlreb}), the assertion
(\ref{intVerRegDG}) for $K=1$ is verified.

\smallskip
The general statement follows by induction, once (\ref{intVerRegDG})
is checked for $K=2$. This will be achieved by adapting the proof of
\cite[Lemma 2.3]{HM03} to the present situation. For the rest of the
proof we abbreviate $P:=P_{m,M,r}(L)$. Let $f\in L^2(X)$ with $\supp
f\subset U_i(B)$ and $\|f\|_{L^2}=1$ be fixed. We consider the set
\begin{align*}
G&:=
\bigl\{y\in
X\,:\,\mbox{dist}(y,U_j(B))<\tfrac12\,\mbox{dist}(U_i(B),U_j(B))\bigr\}
\\&\phantom{:}=
\bigl\{y\in X\,:\,(2^{j-2}+2^{i-1})r<d(x,y)<(5\cdot2^{j-2}-2^{i-1})r\bigr\}
\end{align*}
and analyze
\begin{align*}
\bigl\|\cf_{U_j(B)}P^2f\bigr\|_{L^2}
\leq
\bigl\|P(\cf_G\cdot Pf)\bigr\|_{L^2(U_j(B))}
+
\bigl\|P(\cf_{X\setminus G}\cdot Pf)\bigr\|_{L^2(U_j(B))}\,.
\end{align*}
In order to estimate the first term on the right-hand side, we 
initially exploit the boundedness of $P$ on $L^2(X)$ and then cover
the set $G$ by dyadic annuli in such a way as to enable us to apply
(\ref{rgkldghdhdhgdhhf}):
\begin{align*}
&\bigl\|P(\cf_G\cdot Pf)\bigr\|_{L^2(U_j(B))}
\leqC
\|\cf_G\cdot Pf\|_{L^2}
\leq
\sum_{k=\lfloor\log_2(2^{j-2}+2^{i-1})\rfloor}
^{\lfloor\log_2(5\cdot2^{j-2}-2^{i-1})\rfloor+1}\|\cf_{U_k(B)}\cdot Pf\|_{L^2}
\\&\qquad\leqC
\sum_{k=\lfloor\log_2(2^{j-2}+2^{i-1})\rfloor}
^{\lfloor\log_2(5\cdot2^{j-2}-2^{i-1})\rfloor+1}e^{-b(2^{k-1}-2^{i+2})}\|f\|_{L^2}
\\&\qquad\leq
\bigl((\log_2(5\cdot2^{j-2}-2^{i-1})+3-\log_2(2^{j-2}+2^{i-1})\bigr)
e^{-b((2^{j-2}+2^{i-1})/4-2^{i+2})}
\\&\qquad\leqC
e^{-b(2^{j-4}-2^{i+2})}\,.
\end{align*}
Thanks to (\ref{rgfrejfhlreb}), the latter is bounded by a constant
times $\exp(-b\,2^{j-i})$, as desired.

\smallskip
The second summand $\|P(\cf_{X\setminus G}\cdot Pf)\|_{L^2(U_j(B))}$
can be treated in an analogous manner. One has only to interchange
the sequence of the arguments. At first, one covers $X\setminus G$
by dyadic annuli, so that the off-diagonal estimate
(\ref{rgkldghdhdhgdhhf}) is applicable, and then one utilizes the
boundedness of $P$ on $L^2(X)$ as well as (\ref{rgfrejfhlreb}). This
gives a similar estimate as before and finishes the proof.
\end{spezbew}

\end{document}